\documentclass{amsart}
\usepackage{amssymb}
\usepackage{amsmath}
\usepackage{amsfonts}
\usepackage{fancyhdr}

\newtheorem{theorem}{Theorem}
\theoremstyle{plain}

\newtheorem{corollary}{Corollary}

\newtheorem{definition}{Definition}

\newtheorem{lemma}{Lemma}

\newtheorem{proposition}{Proposition}
\newtheorem{remark}{Remark}

\numberwithin{equation}{section}
\input{tcilatex}
\begin{document}
\title{Non-standard approximations of the It\^{o}-map}
\author{Peter Friz$^{\ast }$, Harald Oberhauser}
\email{P.Friz@statslab.cam.ac.uk}
\address{Department of Pure Mathematics and Mathematical Statistics\\
University of Cambridge\\
Centre for Mathematical Sciences\\
Wilberforce Road\\
Cambridge\\
CB3 0WB}

\begin{abstract}
The Wong-Zakai theorem asserts that ODEs driven by "reasonable" (e.g.
piecewise linear) approximations of Brownian motion converge to the
corresponding Stratonovich stochastic differential equation. With the aid of
rough path analysis, we study "non-reasonable" approximations and go beyond
a well-known criterion of [Ikeda--Watanabe, North Holland 1989] in the sense
that our result applies to perturbations on all levels, exhibiting
additional drift terms involving any iterated Lie brackets of the driving
vector fields. In particular, this applies to the approximations by McShane
('72) and Sussmann ('91). Our approach is not restricted to Brownian driving
signals. At last, these ideas can be used to prove optimality of certain
rough path estimates.
\end{abstract}

\keywords{Iterated Lie Brackets in Limit Processes of Differential
Equations, Rough Paths Analysis\\
$\ast $) Partially supported by a Leverhulme Research Fellowship and\ EPSRC\
Grant EP/E048609/1.}
\maketitle

\section{Preliminaries}

\subsection{Rough differential equations\label{SectionRDEs}}

Let $\alpha \in \left( 0,1\right] $. A weak geometric $\alpha $-H\"{o}lder
rough path $\mathbf{x}$ over $\mathbb{R}^{d}$ is a continuous path on $\left[
0,T\right] $ with values in $G^{\left[ 1/\alpha \right] }\left( \mathbb{R}%
^{d}\right) $, the step\footnote{$\left[ \cdot \right] $ gives the integer
part of a real number.}-$\left[ 1/\alpha \right] $ nilpotent group over $%
\mathbb{R}^{d}$, of finite $\alpha $-H\"{o}lder regularity relative to $d$,
the Carnot-Carath\'{e}odory metric on $G^{\left[ p\right] }\left( \mathbb{R}%
^{d}\right) $, i.e.%
\begin{equation*}
\left\Vert \mathbf{x}\right\Vert _{\alpha \text{-H\"{o}l;}\left[ 0,T\right]
}=\sup_{0\leq s<t\leq T}\frac{d\left( \mathbf{x}_{s},\mathbf{x}_{t}\right) }{%
\left\vert t-s\right\vert ^{\alpha }}<\infty \text{.}
\end{equation*}%
For orientation, let us discuss the case $\alpha \in \left( 1/3,1/2\right) $%
, which covers Brownian motion (for details see \cite%
{FrizVictoir05:ApproximationsEBM},\cite{lyons-98},\cite%
{lyons-caruana-levy-07},\cite{frizvictoirbook}). We realize $G^{2}\left( 
\mathbb{R}^{d}\right) $ as the set of all $\left( a,b\right) \in $ $\mathbb{R%
}^{d}\oplus \mathbb{R}^{d\times d}$ for which $Sym\left( b\right) \equiv
a^{\otimes 2}/2$. (This point of view is natural: a smooth $\mathbb{R}^{d}$%
-valued path $x=\left( x_{t}^{i}\right) _{i=1,\dots ,d}$, enhanced with its
iterated integrals $\int_{0}^{t}\int_{0}^{s}dx_{u}^{i}dx_{s}^{j}$, gives
canonically rise to a $G^{2}\left( \mathbb{R}^{d}\right) $-valued path.).
Given $\left( a,b\right) \in G^{2}\left( \mathbb{R}^{d}\right) $ one gets
rid of the redundant $Sym\left( b\right) $ by $\left( a,b\right) \mapsto $ $%
\left( a,b-a^{\otimes }/2\right) \in \mathbb{R}^{d}\oplus so\left( d\right) $%
. Applied to $x$ enhanced with its iterated integrals over $\left[ 0,t\right]
$ this amounts to look at the path $x$ and its (signed) areas $%
\int_{0}^{t}x_{0,s}^{i}dx_{s}^{j}-\int_{0}^{t}x_{0,s}^{j}dx_{s}^{i},\,\,i,j%
\in \left\{ 1,\dots ,d\right\} $\footnote{%
Given an Interval $I=\left[ a,b\right] ,$ for brevity we write $x_{I}\equiv
x_{a,b}\equiv x_{b}-x_{a}$.}. Without going in too much detail, the group
structure on $G^{2}\left( \mathbb{R}^{d}\right) $ can be identified with the
(truncated) tensor multiplication and is relevant as it allows to relate
algebraically the path and area increments over adjacent intervals; the
mapping $\left( a,b\right) \mapsto $ $\left( a,b-a^{\otimes }/2\right) $
maps the Lie group $G^{2}\left( \mathbb{R}^{d}\right) $ to its Lie algebra;
at last, the Carnot-Carath\'{e}odory metric is defined intrinsically as
(left-)invariant metric on $G^{2}\left( \mathbb{R}^{d}\right) $ and
satisfies $\left\vert a\right\vert +\left\vert b\right\vert ^{1/2}\lesssim
d\left( \left( 0,0\right) ,\left( a,b\right) \right) \lesssim \left\vert
a\right\vert +\left\vert b\right\vert ^{1/2}$.

One can then think of a geometric $\alpha $-H\"{o}lder rough path $\mathbf{x}
$ as a path $x:\left[ 0,T\right] \rightarrow \mathbb{R}^{d}$ enhanced with
its iterated integrals (equivalently: area integrals) although the later
need not make classical sense. For instance, \textit{almost every}\ joint
realization of Brownian motion and L\'{e}vy's area process is a geometric $%
\alpha $-H\"{o}lder rough path. Lyons' theory of rough paths then gives
deterministic meaning to the rough differential equation (RDE)%
\begin{equation*}
\mathrm{d}y=V\left( y\right) \mathrm{d}\mathbf{x,}\text{ }y\left( 0\right)
=y_{0}
\end{equation*}%
for $\mathrm{Lip}^{\Gamma }$-vector fields (in the sense of Stein\footnote{%
i.e.\ a function is Lip$^{\gamma }$ if it is $\left\lfloor \gamma
\right\rfloor $-times ($\left\lfloor .\right\rfloor =\left[ .\right] -1$ on
integers, otherwise equal) differentiable, the $\left\lfloor \gamma
\right\rfloor $th derivative is $\gamma -\left\lfloor \gamma \right\rfloor $%
-H\"{o}lder continuous and the function and all its derivatives are bounded.}%
), $\Gamma >1/\alpha \geq 1$, and we write $y_{t}=\pi \left(
0,y_{0};x\right) _{t}$ for this solution. By considering the space-time
rough path $\mathbf{\tilde{x}}=\left( t,\mathbf{x}\right) $ and $\tilde{V}%
=\left( V_{0},V_{1},...,V_{d}\right) $ one can consider RDEs with drift.
Although well studied \cite{LejayVictoir06:pqRoughPaths}, with a few towards
minimal regularity assumptions on $V_{0}$, we shall need certain "Euler"
estimates \cite{FrizVictoir06:EulerEstimatesforRDEs} for RDEs with drift
which are not available in the current literature. The "Doss--Sussmann
method" (implemented for RDEs in section \ref{secRDEwithDrift}) will provide
a quick route to these estimates.

\subsection{Standard and non-standard approximations\label{Preliminaries2}}

Assume we are given a weak geometric $\alpha $-H\"{o}lder rough path $%
\mathbf{x}$ and a path $\mathbf{p}$ that takes values in the center of $%
G^{N}\left( \mathbb{R}^{d}\right) $, $N$ some integer $N\geq \left[ 1/\alpha %
\right] $ (think of the path $\mathbf{p}$ as a perturbation of our original
path $\mathbf{x}$\textbf{)}. Further assume that $\mathbf{p}$ is $\beta $-H%
\"{o}lder continuous. It is a well-known \cite{lyons-98} that $\mathbf{x}$
can be lifted uniquely to a $\alpha $-H\"{o}lder path $S_{N}\left( \mathbf{x}%
\right) $ with values in $G^{N}\left( \mathbb{R}^{d}\right) $. Then,%
\begin{equation*}
S_{N}\left( \mathbf{x}\right) _{0,\cdot }\otimes \mathbf{p}\in C^{\min
\left( \alpha ,\beta \right) \text{-H\"{o}l}}\left( \left[ 0,T\right]
,G^{N}\left( \mathbb{R}^{d}\right) \right) .
\end{equation*}%
From general facts of such spaces (e.g. \cite%
{FrizVictoir04:NoteonGeomRoughPaths}, relying only on the fact that $%
G^{N}\left( \mathbb{R}^{d}\right) $ is a geodesic space) we can find a
sequence of Lipschitz continuous paths, $x^{n}:\left[ 0,T\right] \rightarrow 
\mathbb{R}^{d}$, so that 
\begin{equation*}
d_{\infty }\left( S_{N}\left( x^{n}\right) ,S_{N}\left( \mathbf{x}\right)
_{0,\cdot }\otimes \mathbf{p}\right) \rightarrow 0\text{ as }n\rightarrow
\infty \text{ and }\sup_{n\in 
\mathbb{N}
}\left\Vert S_{N}\left( x^{n}\right) \right\Vert _{\min \left( \alpha ,\beta
\right) \text{-H\"{o}l}}<\infty \text{.}
\end{equation*}%
(The approximations $x^{n}$ are constructed based on geodesics associated to
the $G^{N}\left( \mathbb{R}^{d}\right) $-valued increments of $S_{N}\left( 
\mathbf{x}\right) _{0,\cdot }\otimes \mathbf{p}$.) By interpolation, it then
follows that for all $\gamma <\min \left( \alpha ,\beta \right) $,%
\begin{equation*}
d_{\gamma \text{-H\"{o}l}}\left( S_{N}\left( x^{n}\right) ,S_{N}\left( 
\mathbf{x}\right) \otimes \mathbf{p}\right) \rightarrow 0\text{ as }%
n\rightarrow \infty \text{ .}
\end{equation*}%
The interest of such a construction is that the limiting behaviour of ODEs 
\textit{driven by the }$x^{n}$, provided $\left[ 1/\gamma \right] \geq N$
and $\Gamma >1/\gamma $, exhibits additional drift behaviour in terms of the
Lie brackets of the driving vector fields and the perturbation $\mathbf{p}$
(cf.\ sections \ref{SecControlledODE} and \ref{secRDEwithDrift}).

We may apply this to Brownian and L\'{e}vy area, i.e.\ $\mathbf{x=B}\left(
\omega \right) =\exp \left( B+A\right) ,$ in which case the approximations $%
x^{n}$ are constructed in a purely deterministic fashion based on the
realization of the \textit{Brownian path and its L\'{e}vy area. }Interesting
as it may be, this is not fully satisfactory as it stands in contrast to
(probabilistic)\ non-standard approximation results (McShane \cite%
{McShane72:SDEandModels}, Sussmann \cite{Sussmann91:WongZakaiWithDrift},
...) which have the desirable property that the approximations depend only
on (finitely many points of) the Brownian path. The first aim of this paper
is to give a criterion that covers all these examples in a flexible
frame-work of random rough paths. En passant, this allows for a painfree
extension to various non-Brownian driving signals (cf. remark \ref{various}%
). We then give a rigorous analysis on how to translate RDEs driven by $%
S_{N}\left( \mathbf{x}\right) \otimes \mathbf{p}$ to RDEs driven by $\mathbf{%
x}$ in which see the appearance of additional drift vector fields, obtained
as contraction of iterated Lie brackets of $V=\left( V_{1},\dots
,V_{d}\right) $ and the components of $\mathbf{p}$. (At this stage, we need
a good quantitative understanding of RDEs with drift). At last, as a
spin-off of these ideas, we show in section \ref{SectOpt} optimality of
certain rough path estimates, answering a question that left was open in 
\cite{CassFrizVictoir:WienerFunct}.

\section{(Non-Standard) Approximations\label{secApprox}}

\subsection{Criterion for convergence to a "non-standard" limit}

\begin{theorem}
\label{ThmSussmannAbstract}Let $\alpha ,\beta \in \left( 0,1\right] $ and $%
N\geq \left[ 1/\gamma \right] $ for $\gamma :=\min \left( \alpha ,\beta
\right) $. Also, let $\mathbf{x}\in C^{\alpha \text{-H\"{o}l}}\left( \left[
0,T\right] ,G^{\left[ 1/\alpha \right] }\left( \mathbb{R}^{d}\right) \right) 
$, write\ $x=\pi _{1}\left( \mathbf{x}\right) $\ for its projection to a
path with values in $\mathbb{R}^{d}$, and assume that there exists a
sequence of dissections $\left( D_{n}\right) =\left( t_{i}^{n}:i\right) $ of 
$\left[ 0,T\right] ,$ such that%
\begin{equation*}
\sup_{n\in 
\mathbb{N}
}\left\Vert S_{\left[ 1/\alpha \right] }\left( x^{D_{n}}\right) \right\Vert
_{\alpha \text{-H\"{o}l}}<\infty \text{ and }d_{\infty }\left( S_{\left[
1/\alpha \right] }\left( x^{D_{n}}\right) ,\mathbf{x}\right) \rightarrow
_{n\rightarrow \infty }0,
\end{equation*}%
where we write $x^{D_{n}}$ for the usual piecewise linear approximation to $%
x $ based on$D_{n}$.

Let $\left( x^{n}\right) \subset C^{1\text{-H\"{o}l}}\left( \left[ 0,T\right]
,\mathbb{R}^{d}\right) $ such that, for all $t\in D_{n}$,%
\begin{equation*}
\mathbf{p}_{t}^{n}:=S_{N}\left( x^{n}\right) _{0,t}\otimes S_{N}\left(
x^{D_{n}}\right) _{0,t}^{-1}
\end{equation*}%
takes values in the center of $G^{N}\left( \mathbb{R}^{d}\right) $.%
\renewcommand{\theenumi}{\roman{enumi}}

\begin{enumerate}
\item If there exists $c_{1},c_{2},c_{3}\in \lbrack 1,\infty )$ such that
for all $t_{i}^{n},t_{i+1}^{n}\in D_{n}$%
\begin{eqnarray*}
\left\vert x^{n}\right\vert _{1\text{-H\"{o}l};\left[ t_{i}^{n},t_{i+1}^{n}%
\right] } &\leq &c_{1}\left\vert x^{D_{n}}\right\vert _{1\text{-H\"{o}l};%
\left[ t_{i}^{n},t_{i+1}^{n}\right] }+c_{2}\left\vert
t_{i+1}^{n}-t_{i}^{n}\right\vert ^{\beta -1}\text{ and} \\
\left\Vert \mathbf{p}_{s,t}^{n}\right\Vert &\leq &c_{3}\left\vert
t-s\right\vert ^{\beta }\text{ for all }s,t\in D_{n}.
\end{eqnarray*}%
Then there exists a $C=C\left( \alpha ,\beta ,c_{1}\left\Vert \mathbf{x}%
\right\Vert _{\alpha \text{-H\"{o}l}},c_{2},T,N\right) $ such that%
\begin{equation*}
\sup_{n\in 
\mathbb{N}
}\left\Vert S_{N}\left( x^{n}\right) \right\Vert _{\gamma \text{-H\"{o}l}%
}\leq C\left( \sup_{n\in 
\mathbb{N}
}\left\Vert S_{\left[ 1/\alpha \right] }\left( x^{D_{n}}\right) \right\Vert
_{\alpha \text{-H\"{o}l}}+c_{3}+1\right) <\infty
\end{equation*}

\item If $\mathbf{p}_{t}^{n}\rightarrow \mathbf{p}_{t}$ for all $t\in \cup
_{n}D_{n}$ and $\cup _{n}D_{n}$ is dense in $\left[ 0,T\right] $ then $%
\mathbf{p}$ extends to a continuous (in fact, $\beta $-H\"{o}lder
continuous) path with values in the center of $G^{N}\left( \mathbb{R}%
^{d}\right) $ and for all $t\in \left[ 0,T\right] $, 
\begin{eqnarray*}
d\left( S_{N}\left( x^{n}\right) _{0,t},S_{N}\left( \mathbf{x}\right)
_{0,t}\otimes \mathbf{p}_{0,t}\right) &\leq &d\left( S_{N}\left(
x^{D_{n}}\right) _{0,t},S_{N}\left( \mathbf{x}\right) _{0,t}\right) \\
&&+d\left( \mathbf{p}_{0,t}^{n},\mathbf{p}_{0,t}\right) \text{.}
\end{eqnarray*}

and converges to $0$ as $n\rightarrow \infty $.
\end{enumerate}

In particular, if the assumptions of both (i) and (ii) are met then%
\begin{eqnarray*}
d_{\infty }\left( S_{N}\left( x^{n}\right) ,S_{N}\left( \mathbf{x}\right)
\otimes \mathbf{p}\right) &\rightarrow &_{n\rightarrow \infty }0, \\
\sup_{n\in 
\mathbb{N}
}\left\Vert S_{N}\left( x^{n}\right) \right\Vert _{\gamma \text{-H\"{o}l}}
&<&\infty \text{.}
\end{eqnarray*}%
and by interpolation, for all $\gamma ^{\prime }<\gamma $,%
\begin{equation*}
d_{\gamma ^{\prime }\text{-H\"{o}l}}\left( S_{N}\left( x^{n}\right)
,S_{N}\left( \mathbf{x}\right) \otimes \mathbf{p}\right) \rightarrow
_{n\rightarrow \infty }0\text{.}
\end{equation*}
\end{theorem}

\begin{proof}
(i) Take $s<t$ in $\left[ 0,T\right] $. If $s,t\in \left[
t_{i}^{n},t_{i+1}^{n}\right] $ we have by our assumption on $\left\vert
x^{n}\right\vert _{1\text{-H\"{o}l};\left[ t_{i},t_{i+1}\right] }$%
\begin{eqnarray*}
\left\Vert S_{N}\left( x^{n}\right) _{s,t}\right\Vert &\leq &\left\vert
t-s\right\vert \left\Vert S_{N}\left( x^{n}\right) \right\Vert _{1\text{-H%
\"{o}l};\left[ t_{i}^{n},t_{i+1}^{n}\right] } \\
&=&\left\vert t-s\right\vert \left\vert x^{n}\right\vert _{1\text{-H\"{o}l};%
\left[ t_{i}^{n},t_{i+1}^{n}\right] } \\
&\leq &\left\vert t-s\right\vert \left\{ c_{1}\left\vert
x_{t_{i}^{n},t_{i+1}^{n}}\right\vert +c_{2}\left\vert
t_{i+1}^{n}-t_{i}^{n}\right\vert ^{\beta -1}\right\} \\
&\leq &\left\vert t-s\right\vert \left\{ c_{1}\left\vert x\right\vert
_{\alpha \text{-H\"{o}l}}\left\vert t_{i+1}^{n}-t_{i}^{n}\right\vert
^{\alpha }+c_{2}\left\vert t_{i+1}^{n}-t_{i}^{n}\right\vert ^{\beta
-1}\right\} \\
&\leq &\left\vert t-s\right\vert ^{\gamma }C_{0},
\end{eqnarray*}%
with $C_{0}=C_{0}\left( \alpha ,\beta ,c_{1}\left\Vert \mathbf{x}\right\Vert
_{\alpha \text{-H\"{o}l}},c_{2},T\right) $. Otherwise we can find $%
t_{i}^{n}\leq t_{j}^{n}$ so that $s\leq t_{i}^{n}\leq t_{j}^{n}\leq t$ and 
\begin{equation*}
\left\Vert S_{N}\left( x^{n}\right) _{s,t}\right\Vert \leq \left\vert
t-s\right\vert ^{\gamma }2C_{0}+\left\Vert S_{N}\left( x^{n}\right)
_{t_{i}^{n},t_{j}^{n}}\right\Vert .
\end{equation*}%
Using estimates for the Lyons-lift $\mathbf{x}\mapsto S_{N}\left( \mathbf{x}%
\right) $, \cite[theorem 2.2.1]{lyons-98}, we can further estimate%
\begin{eqnarray*}
\left\Vert S_{N}\left( x^{n}\right) _{t_{i}^{n},t_{j}^{n}}\right\Vert &\leq
&\left\Vert S_{N}\left( x^{D_{n}}\right) _{t_{i}^{n},t_{j}^{n}}\right\Vert
+\left\Vert \mathbf{p}_{t_{i}^{n},t_{j}^{n}}^{n}\right\Vert \\
&\leq &c_{N,\alpha }\left\Vert S_{\left[ 1/\alpha \right] }\left(
x^{D_{n}}\right) \right\Vert _{\alpha \text{-H\"{o}l}}\left\vert
t_{j}^{n}-t_{i}^{n}\right\vert ^{\alpha }+c_{3}\left\vert
t_{j}^{n}-t_{i}^{n}\right\vert ^{\beta } \\
&\leq &(c_{N,\alpha }\left\Vert S_{\left[ 1/\alpha \right] }\left(
x^{D_{n}}\right) \right\Vert _{\alpha \text{-H\"{o}l}}+c_{3})\left\vert
t-s\right\vert ^{\gamma }
\end{eqnarray*}%
and, since $\sup_{n}\left\Vert S_{\left[ 1/\alpha \right] }\left(
x^{D_{n}}\right) \right\Vert _{\alpha \text{-H\"{o}l}}<\infty $ by
assumption, the proof of the uniform H\"{o}lder bound is finished.\newline
(ii) By assumption, $\mathbf{p}^{n}$ is uniformly $\beta $-H\"{o}lder. By a
standard Arzela-Ascoli type argument, it is clear that every pointwise limit
(if only on the dense set $\cup _{n}D_{n}$) is a uniform limit. $\beta $-H%
\"{o}lder regularity is preserved in this limit and so $\mathbf{p}$ is $%
\beta $-H\"{o}lder itself. Since $t\in \cup _{n}D_{n},\mathbf{p}_{t}^{n}$
take values in the step-$N$ center for all $t\in D_{n}$, $\cup _{n}D_{n}$ is
dense, and so it is easy to see that $\mathbf{p}_{\cdot }$ takes values in
the step-$N$ center for all $t\in \left[ 0,T\right] $. Now take $t\in D_{n}$%
. Since elements in the center commute with all elements in $G^{N}\left( 
\mathbb{R}^{d}\right) $ we have%
\begin{eqnarray*}
&&d\left( S_{N}\left( x^{n}\right) _{0,t},S_{N}\left( \mathbf{x}\right)
_{0,t}\otimes \mathbf{p}_{0,t}\right) \\
&=&\left\Vert S_{N}\left( x^{n}\right) _{0,t}^{-1}\otimes S_{N}\left(
x^{D_{n}}\right) _{0,t}\otimes \mathbf{p}_{0,t}^{n}\otimes S_{N}\left(
x^{D_{n}}\right) _{0,t}^{-1}\otimes S_{N}\left( \mathbf{x}\right)
_{0,t}\otimes \left( \mathbf{p}_{0,t}^{n}\right) ^{-1}\otimes \mathbf{p}%
_{0,t}\right\Vert \\
&=&\left\Vert S_{N}\left( x^{D_{n}}\right) _{0,t}^{-1}\otimes S_{N}\left( 
\mathbf{x}\right) _{0,t}\otimes \left( \mathbf{p}_{0,t}^{n}\right)
^{-1}\otimes \mathbf{p}_{0,t}\right\Vert \\
&\leq &d\left( S_{N}\left( x^{D_{n}}\right) _{0,t},S_{N}\left( \mathbf{x}%
\right) _{0,t}\right) +d\left( \mathbf{p}_{0,t}^{n},\mathbf{p}_{0,t}\right)
\end{eqnarray*}%
On the other hand, given an arbitrary element $t\in \left[ 0,T\right] $ we
can take $t^{n}$ to be the closest neighbour in $D_{n}$ and so%
\begin{eqnarray*}
&&d\left( S_{N}\left( x^{n}\right) _{0,t},S_{N}\left( \mathbf{x}\right)
_{0,t}\otimes \mathbf{p}_{0,t}\right) \\
&=&d\left( S_{N}\left( x^{D_{n}}\right) _{0,t},S_{N}\left( \mathbf{x}\right)
_{0,t}\right) +2d\left( \mathbf{p}_{0,t}^{n},\mathbf{p}_{0,t}\right) \\
&&+d\left( S_{N}\left( \mathbf{x}\right) _{0,t}^{-1}\otimes S_{N}\left(
x^{n}\right) _{0,t},S\left( \mathbf{x}\right) _{0,t^{n}}^{-1}\otimes
S_{N}\left( x^{n}\right) _{0,t^{n}}\right) .
\end{eqnarray*}%
From the assumptions and H\"{o}lder (resp.\ uniform H\"{o}lder) continuity
of $S\left( \mathbf{x}\right) $ (resp. $S_{N}\left( x^{n}\right) $) we see
that $d\left( S_{N}\left( x^{n}\right) _{0,t},S_{N}\left( \mathbf{x}\right)
_{0,t}\otimes \mathbf{p}_{0,t}\right) \rightarrow 0$, as required.
\end{proof}

\begin{corollary}
\label{CorMcShaneAbstract}Let $\alpha ,\beta \in \left( 0,1\right] $, $N\geq %
\left[ 1/\min \left( \alpha ,\beta \right) \right] $. Also, let $\mathbf{X}%
\left( \omega \right) \in C_{0}^{\alpha \text{-H\"{o}l}}\left( \left[ 0,T%
\right] ,G^{\left[ 1/\alpha \right] }\left( \mathbb{R}^{d}\right) \right) $,
write\ $X=\pi _{1}\left( \mathbf{X}\right) $ \ for its projection to a
process with values in $\mathbb{R}^{d}$, and assume that%
\begin{eqnarray*}
\forall q\in 
\mathbb{N}
:\sup_{n\in 
\mathbb{N}
}\left\vert \left\Vert S_{\left[ 1/\alpha \right] }\left( X^{D_{n}}\right)
\right\Vert _{\alpha \text{-H\"{o}l}}\right\vert _{L^{q}} &<&\infty \text{ }
\\
d_{\infty }\left( S_{\left[ 1/\alpha \right] }\left( X^{D_{n}}\right) ,%
\mathbf{X}\right) &\rightarrow &0\text{ in probability as }n\rightarrow
\infty \text{.}
\end{eqnarray*}%
Let $\left( X^{n}\left( \omega \right) \right) _{n\in 
\mathbb{N}
}\subset C^{1\text{-H\"{o}l}}\left( \left[ 0,1\right] ,\mathbb{R}^{d}\right) 
$ such that, for all $\omega $ and all $t\in D_{n}$%
\begin{equation*}
\mathbf{P}_{t}^{n}\left( \omega \right) :=S_{N}\left( X^{n}\right)
_{0,t}\otimes S_{N}\left( X^{D_{n}}\right) _{0,t}^{-1}
\end{equation*}%
takes values in the center of $G^{N}\left( \mathbb{R}^{d}\right) $.%
\renewcommand{\theenumi}{\roman{enumi}}

\begin{enumerate}
\item If there exists $c_{1},c_{2},c_{3}\in \lbrack 1,\infty )$ such that
for all $t_{i}^{n},t_{i+1}^{n}\in D_{n}$, all $\omega $ and all $q\in
\lbrack 1,\infty ),$%
\begin{eqnarray*}
\left\vert X^{n}\right\vert _{1\text{-H\"{o}l};\left[ t_{i}^{n},t_{i+1}^{n}%
\right] } &\leq &c_{1}\left\vert X^{D_{n}}\right\vert _{1\text{-H\"{o}l};%
\left[ t_{i}^{n},t_{i+1}^{n}\right] }+c_{2}\left\vert
t_{i+1}^{n}-t_{i}^{n}\right\vert ^{\beta -1} \\
\left\vert \left\Vert \mathbf{P}_{s,t}^{n}\right\Vert \right\vert
_{L^{q}\left( \mathbb{P}\right) } &\leq &c_{q}\left\vert t-s\right\vert
^{\beta }\text{ for all }s,t\in \left[ 0,T\right] \text{ }
\end{eqnarray*}%
then, for all $\gamma <\min \left( \alpha ,\beta \right) $,%
\begin{equation*}
\forall q:\sup_{n\in 
\mathbb{N}
}\left\vert \left\Vert S_{N}\left( x^{n}\right) \right\Vert _{\gamma \text{-H%
\"{o}l}}\right\vert _{L^{q}\left( \mathbb{P}\right) }<\infty .
\end{equation*}

\item If $\mathbf{P}_{t}^{n}\rightarrow \mathbf{P}_{t}$ in probability for
all $t\in \cup _{n\in 
\mathbb{N}
}D_{n}$ and $\cup _{n\in 
\mathbb{N}
}D_{n}$ is dense in $\left[ 0,T\right] $ then%
\begin{equation*}
d\left( S_{N}\left( X^{n}\right) _{0,t},S_{N}\left( \mathbf{X}\right)
_{0,t}\otimes \mathbf{P}_{0,t}\right) \rightarrow 0\text{ in probability }
\end{equation*}
\end{enumerate}

In particular, if the assumptions of both (i) and (ii) are met then, for all 
$\gamma <\min \left( \alpha ,\beta \right) $,%
\begin{equation*}
d_{\gamma \text{-H\"{o}l}}\left( S_{N}\left( X^{n}\right) ,S_{N}\left( 
\mathbf{X}\right) \otimes \mathbf{P}\right) \rightarrow 0\text{ in }L^{q}%
\text{ for all }q\in \lbrack 1,\infty ).
\end{equation*}
\end{corollary}

\begin{proof}
(i) By a standard Garsia-Rodemich-Rumsey or Kolmogorov argument, the
assumption on $\left\vert \left\Vert \mathbf{P}_{s,t}^{n}\right\Vert
\right\vert _{L^{q}\left( \mathbb{P}\right) }$ implies, for any $\tilde{\beta%
}<\beta $, the existence of $C_{3}\in L^{q}$ for all $q\in \lbrack 1,\infty
) $ so that 
\begin{equation*}
\forall s<t\text{ in }\left[ 0,T\right] :\left\Vert \mathbf{P}%
_{s,t}^{n}\right\Vert \leq C_{3}\left( \omega \right) \left\vert
t-s\right\vert ^{\tilde{\beta}}
\end{equation*}%
We apply theorem \ref{ThmSussmannAbstract} with $\tilde{\beta}$ instead of $%
\beta $ and learn that there exists a deterministic constant $C$ such that%
\begin{equation*}
\sup_{n}\left\Vert S_{N}\left( X^{n}\right) \right\Vert _{\min \left( \alpha
,\tilde{\beta}\right) \text{-H\"{o}l}}\leq C\left( \sup_{n}\left\Vert S_{%
\left[ 1/\alpha \right] }\left( X^{D_{n}}\right) \right\Vert _{\alpha \text{%
-H\"{o}l}}+1+C_{3}\right) .
\end{equation*}%
Taking $L^{q}$-norms finishes the uniform $L^{q}$-bound. (We take $\tilde{%
\beta}$ large enough so that $\min \left( \alpha ,\tilde{\beta}\right)
>\gamma $.)

(ii)\ From theorem \ref{ThmSussmannAbstract}%
\begin{equation*}
d\left( S_{N}\left( X^{n}\right) _{0,t},S_{N}\left( \mathbf{X}\right)
_{0,t}\otimes \mathbf{P}_{0,t}\right) \leq d\left( S_{N}\left(
X^{D_{n}}\right) _{0,t},S_{N}\left( \mathbf{X}\right) _{0,t}\right) +d\left( 
\mathbf{P}_{0,t}^{n},\mathbf{P}_{0,t}\right)
\end{equation*}%
which, from the assumptions, obviously converges to $0$ (in probability) for
every fixed $t\in \cup _{n}D_{n}$. From general facts, of $L^{q}$%
-convergence of rough paths (cf.\ \cite[Appendix]{FrizVictoir07:GaussRP1})
this implies the claimed convergence. (Inspection of the proof shows that
convergence in probability for all $t$ in a dense set of $\left[ 0,T\right] $
is enough.)
\end{proof}

\begin{remark}
\label{various}Although at first sight technical, our assumptions are fairly
natural: firstly, we restrict our attention to (random) H\"{o}lder rough
paths $\mathbf{X}$ which are the limit of "their (lifted) piecewise linear
approximations". This covers the bulk of stochastic processes which admit a
lift to a rough path including semi-martingales \cite{CoutinLejay05,
FrizVictoir06:BDGforEnhancedMartingales}, fractional Brownian motion with $%
H>1/4$ and many other Gaussian processes, \cite{coutin-qian-02,
FrizVictoir07:GaussRP1}, as well as Markov processes with uniformly elliptic
generator in divergence form \cite{stroock-88},\cite{lejay-02},\cite%
{FrizVictoir06:OnuniformlysubellipiticOperators}.\newline
Secondly, the assumptions on $X^{n}$ and $\mathbf{P}^{n}$ guarantee that $%
X^{n}$ remains, at $\min \left( \alpha ,\beta \right) $-H\"{o}lder scale,
comparable to the piecewise linear apprxoximations. In particular, the
assumption on $\left\vert X^{n}\right\vert _{1\text{-H\"{o}l};\left[
t_{i}^{n},t_{i+1}^{n}\right] }=\left\vert \dot{X}^{n}\right\vert _{\infty ;%
\left[ t_{i}^{n},t_{i+1}^{n}\right] }$ is easy to verify in all examples
below. The intuition is that, if we assume that $X^{n}$ runs at constant
speed over any interval $I=\left[ t_{i}^{n},t_{i+1}^{n}\right] $, $%
D_{n}=\left( t_{i}^{n}\right) $, it is equivalent to saying that%
\begin{eqnarray*}
\text{length}\left( X^{n}|_{I}\right)  &\leq &c_{1}\text{length}\left(
X^{D_{n}}|_{I}\right) +c_{2}\left\vert I\right\vert ^{\beta } \\
( &=&c_{1}\left\vert X_{t_{i}^{n},t_{i+1}^{n}}\right\vert +c_{2}\left\vert
t_{i+1}^{n}-t_{i}^{n}\right\vert ^{\beta })
\end{eqnarray*}
\end{remark}

\subsection{Examples}

\subsubsection{Sussmann \protect\cite{Sussmann91:WongZakaiWithDrift}}

Take any sequence of dissection of $\left[ 0,T\right] ,$ say $\left(
D_{n}\right) $ with mesh $\left\vert D_{n}\right\vert \rightarrow 0$ and
think of $\mathbf{X}\left( \omega \right) $ as Brownian motion plus L\'{e}%
vy's area so that $\pi _{1}\left( \mathbf{X}\right) =X$ is $d$-dimensional
Brownian motion. The piecewise linear approximation $X^{D_{n}}$ is nothing
but the repeated concatentation of linear chords connecting the points $%
\left( X_{t}:t\in D_{n}\right) $. For some fixed $\mathbf{v}\in \mathfrak{g}%
^{N}\left( \mathbb{R}^{d}\right) \cap \left( \mathbb{R}^{d}\right) ^{\otimes
N}$, $N\in \left\{ 2,3,\dots \right\} $ we now construct \textit{Sussmann's
non-standard approximation} $X^{n}$ as (repeated) concatenation of linear
chords and "geodesic loops". First, we require $X^{n}\left( t\right)
=X\left( t\right) $ for all $t\in D_{n}=\left( t_{i}^{n}:i\right) $. For
intermediate times, i.e.\ $t\in \left( t_{i-1}^{n},t_{i}^{n}\right) $ for
some $i$ we proceed as follows: For $t\in \lbrack t_{i-1}^{n},\left(
t_{i-1}^{n}+t_{i}^{n}\right) /2]$ we run linearly and at constant speed from 
$X\left( t_{i-1}^{n}\right) $ such as to reach $X\left( t_{i}^{n}\right) $
by time $\left( t_{i-1}^{n}+t_{i}^{n}\right) /2$. (This is the usual linear
interpolation between $X\left( t_{i-1}^{n}\right) $ and $X\left(
t_{i}^{n}\right) $ but run at double speed.) This leaves us with the
interval $[\left( t_{i-1}^{n}+t_{i}^{n}\right) /2,t_{i}^{n}]$ for other
purposes and we run, starting at $x\left( t_{i}^{n}\right) \in \mathbb{R}%
^{d} $, through a "geodesic" $\xi :[\left( t_{i-1}^{n}+t_{i}^{n}\right)
/2,t_{i}^{n}]\rightarrow \mathbb{R}^{d}$ associated to $\exp \left( \mathbf{v%
}/\left\vert t_{i}^{n}-t_{i-1}^{n}\right\vert \right) \in G^{N}\left( 
\mathbb{R}^{d}\right) $.\footnote{%
A path $\xi :[a,b]\rightarrow \mathbb{R}^{d}$ is a geodesic associated to $%
\mathbf{g}\in G^{N}\left( \mathbb{R}^{d}\right) $ if $S_{N}\left( \xi
\right) _{a,b}=\mathbf{g}$ and $\xi $ has lenght equal to the
Carnot-Caratheodory norm of $g$. See \cite%
{FrizVictoir04:NoteonGeomRoughPaths} and the references therein.} Since $N>1$%
, $\pi _{1}\left( \exp \left( \mathbf{v}/\left\vert
t_{i}^{n}-t_{i-1}^{n}\right\vert \right) \right) =0$ and so this geodesic
path returns to its starting point in $\mathbb{R}^{d}$; in particular%
\begin{equation*}
X^{n}\left( \left( t_{i-1}^{n}+t_{i}^{n}\right) /2\right) =X^{n}\left(
t_{i}^{n}\right) =X\left( t_{i}^{n}\right) .
\end{equation*}%
It is easy to see (via Chen's theorem) that this approximation satisfies the
assumptions of corollary \ref{CorMcShaneAbstract} with%
\begin{equation*}
\mathbf{P}_{s,t}^{n}:=S_{N}\left( x^{n}\right) _{s,t}\otimes S_{N}\left(
x^{D_{n}}\right) _{s,t}^{-1}=e^{\mathbf{v}\left( t-s\right) }\text{ }\forall
s,t\in D_{n},
\end{equation*}%
(so that $\left\vert \mathbf{P}_{s,t}^{n}\right\vert _{L^{q}}=\left\vert 
\mathbf{P}_{s,t}^{n}\right\vert \lesssim \left\vert t-s\right\vert ^{1/N%
\text{ }}$ first for all $s,t\in D_{n}$ and then, easy to see, for all $s,t$%
) and deterministic limit $\mathbf{P}_{0,t}=e^{\mathbf{v}t},$ $\beta =1/N$.
Indeed, the length of $x^{n}$ over any interval $I=\left[
t_{i-1}^{n},t_{i}^{n}\right] $ is obviously bounded by the length of the
corresponding linear chord plus the length of the geodesic associated to $%
\exp \left( \mathbf{v}/2^{n}\right) =\exp \left( \mathbf{v}/\left\vert
I\right\vert \right) $, which is precisely equal to%
\begin{equation*}
\left\Vert \exp \left( \mathbf{v}/\left\vert I\right\vert \right)
\right\Vert =\left\vert I\right\vert ^{1/N}\left\Vert \exp \left( \mathbf{v}%
\right) \right\Vert =:c_{2}\left\vert I\right\vert ^{1/N}.
\end{equation*}%
Obviously, nothing here is specific to Brownian motion. An application of
corollary \ref{CorMcShaneAbstract} gives the following convergence result
which, when applied to Brownian motion and L\'{e}vy area and in conjunction
with theorem \ref{ThmSussmannAbstract} below, implies Sussmann's
non-standard approximation result for stochastic differential equations.

\begin{proposition}[Rough path convergence of Sussmann's approximation]
Let $\mathbf{X}\left( \omega \right) \in C_{0}^{\alpha \text{-H\"{o}l}%
}\left( \left[ 0,T\right] ,G^{\left[ 1/\alpha \right] }\left( \mathbb{R}%
^{d}\right) \right) $ be a (random, $\alpha $-H\"{o}lder rough path) which
is the limit of its piecewise linear approximation (as detailed in corollary %
\ref{CorMcShaneAbstract}). Then, for any $\gamma <\min \left( \alpha
,1/N\right) $ we have%
\begin{equation*}
d_{\gamma \text{-H\"{o}l}}\left( S_{N}\left( X^{n}\right) ,S_{N}\left( 
\mathbf{X}\right) \otimes e^{\mathbf{v\cdot }}\right) \rightarrow 0\text{ in 
}L^{q}\text{ for all }q\in \lbrack 1,\infty ).
\end{equation*}
\end{proposition}

Strictly speaking, his construction did not rely on the concept of geodescis
associated to $\exp \left( \mathbf{v}/2^{n}\right) \in G^{N}\left( \mathbb{R}%
^{d}\right) $. Since $\exp \left( \mathbf{v}/2^{n}\right) $ is an element of
the center of the group he can give a reasonably simple inductive
construction of a piecewise linear "approximate geodesic" which is also seen
to satisfy the assumptions of our theorem. We also note that he only
discusses dyadic approximations and obtains on a.s. convergence result
(which can be obtained by a direct application of theorem \ref%
{ThmSussmannAbstract}). In conjunction with theorem \ref{PropRDEwithDrift}
below, we then recover the main result of \cite%
{Sussmann91:WongZakaiWithDrift}.

\subsubsection{McShane \protect\cite{McShane72:SDEandModels}}

Given $x\in C\left( \left[ 0,T\right] ,\mathbb{R}^{2}\right) $, an \textit{%
interpolation function} $\phi =\left( \phi ^{1},\phi ^{2}\right) \in
C^{1}\left( \left[ 0,1\right] ,\mathbb{R}^{2}\right) $ with $\phi \left(
0\right) =\left( 0,0\right) $ and $\phi \left( 1\right) =\left( 1,1\right) $
and a dissection $D=D_{n}=\left\{ t_{i}\right\} $ of $\left[ 0,T\right] $ we
define the \textit{McShane interpolation }$x^{n}\in C\left( \left[ 0,T\right]
,\mathbb{R}^{2}\right) $ by 
\begin{equation*}
x_{t}^{n;i}:=x_{t_{D}}^{i}+\phi ^{\Delta \left( t,i\right) }\left( \frac{%
t-t_{D}}{t^{D}-t_{D}}\right) x_{t_{D},t^{D}}^{i},\text{\thinspace \thinspace
\thinspace }i=1,2.
\end{equation*}%
The points $t_{D},t^{D}\in D$ denote the left-, resp.\ right-, neighbouring
points of $t$ in the dissection and 
\begin{equation*}
\Delta \left( t,i\right) :=\left\{ 
\begin{array}{ccc}
i & \text{, if} & x_{t_{D},t^{D}}^{1}x_{t_{D},t^{D}}^{2}\geq 0 \\ 
3-i & \text{, if} & x_{t_{D},t^{D}}^{1}x_{t_{D},t^{D}}^{2}<0%
\end{array}%
\right.
\end{equation*}%
As a simple consequence of this definition, for $u<v$ in $\left[
t_{i},t_{i+1}\right] $%
\begin{eqnarray*}
S_{2}\left( x^{n}\right) _{u,v} &\equiv &\exp \left(
x_{u,v}^{n}+A_{u,v}^{n}\right) \\
&=&\exp \left( x_{u,v}+\left\vert x_{t_{i},t_{i+1}}^{1}\right\vert
\left\vert x_{t_{i},t_{i+1}}^{2}\right\vert A^{\phi }\left( \frac{u-t_{i}}{%
t_{i+1}-t_{i}},\frac{v-t_{i}}{t_{i+1}-t_{i}}\right) \right)
\end{eqnarray*}%
where $A^{\phi }\left( u,v\right) \equiv A_{u,v}^{\phi }$ is the area
increment of $\phi $ over $\left[ u,v\right] \subset \left[ 0,1\right] $.%
\newline
Consider now $\mathbf{X}\left( \omega \right) =\mathbf{B}\left( \omega
\right) =\exp \left( B+A\right) \in C_{0}^{\alpha \text{-H\"{o}l}}\left( %
\left[ 0,1\right] ,G^{\left[ 1/\alpha \right] }\left( \mathbb{R}^{2}\right)
\right) $ with $\alpha \in \left( 1/3,1/2\right) $ and take any $\left(
D_{n}\right) _{n\in 
\mathbb{N}
}$ with $\left\vert D_{n}\right\vert \rightarrow 0$. (We know, e.g. from 
\cite{FrizVictoir07:GaussRP1}, that lifted piecewise linear approximations, $%
S_{2}\left( B^{D_{n}}\right) $ converges to $\mathbf{B}$ in $1/\alpha $-H%
\"{o}lder rough path topology and in $L^{q}$ for all $q$.) It is easy to see
(via Chen's theorem) that McShane's approximation to two-dimensional
Brownian motion satisfies the assumptions of corollary \ref%
{CorMcShaneAbstract} with $\beta =1/2$, $N=2$. Indeed, writing $S_{2}\left(
B^{n}\right) =\exp \left( B^{n}+A^{n}\right) $ it is clear that for any $s<t$%
\begin{equation*}
\mathbf{P}_{s,t}^{n}=\exp \left( \left\vert x_{t_{D},t^{D}}^{1}\right\vert
\left\vert x_{t_{D},t^{D}}^{2}\right\vert \times A^{\phi }\left( \frac{%
s-t_{D}}{t^{D}-t_{D}},\frac{t-t_{D}}{t^{D}-t_{D}}\right) \right) \text{ with 
}D=D_{n}
\end{equation*}%
and for $t_{i}<t_{j}$ elements of the dissection $D=D_{n}$,%
\begin{equation*}
\mathbf{P}_{t_{i},t_{j}}^{n}=\exp \left( A_{0,1}^{\phi
}\sum_{k=i+1}^{j}\left\vert B_{t_{k},t_{k+1}}^{1}\right\vert \left\vert
B_{t_{k},t_{k+1}}^{2}\right\vert \right) .
\end{equation*}%
It is straight-forward to see that $\sum_{k=i+1}^{j}\left\vert
B_{t_{k},t_{k+1}}^{1}\right\vert \left\vert B_{t_{k},t_{k+1}}^{2}\right\vert 
$ converges, in $L^{2}$ say, to its mean%
\begin{equation*}
\frac{2}{\pi }\sum_{k=i+1}^{j}\left( t_{k+1}-t_{k}\right) =\frac{2}{\pi }%
\left\vert t_{j}-t_{i}\right\vert
\end{equation*}%
while $\left\Vert P_{t_{i},t_{j}}^{n}\right\Vert _{L^{q}}\leq \tilde{c}%
_{q}\left\vert t_{j}-t_{i}\right\vert ^{1/2}$ follows directly from $%
\left\Vert P_{t_{i},t_{j}}^{n}\right\Vert \sim \left(
\sum_{k=i+1}^{j}\left\vert B_{t_{k},t_{k+1}}^{1}\right\vert \left\vert
B_{t_{k},t_{k+1}}^{2}\right\vert \right) ^{1/2}$ and 
\begin{equation*}
\left\vert \sum_{k=i+1}^{j}\left\vert B_{t_{k},t_{k+1}}^{1}\right\vert
\left\vert B_{t_{k},t_{k+1}}^{2}\right\vert \right\vert _{L^{q}}\leq
\sum_{k=i+1}^{j}\left\vert B_{t_{k},t_{k+1}}^{1}\right\vert
_{L^{q}}\left\vert B_{t_{k},t_{k+1}}^{2}\right\vert _{L^{q}}=c_{q}\left\vert
t_{j}-t_{i}\right\vert .
\end{equation*}%
In fact, $\left\Vert P_{s,t}^{n}\right\Vert _{L^{q}}\leq \tilde{c}%
_{q}\left\vert t-s\right\vert ^{1/2}$ for all $s,t$ since for $u<v$ in $%
\left[ t_{i},t_{i+1}\right] $%
\begin{eqnarray*}
\left\Vert P_{u,v}^{n}\right\Vert _{L^{q}} &=&\left( \mathbb{E}\left\vert
\left\vert x_{t_{i},t_{i+1}}^{1}\right\vert \left\vert
x_{t_{i},t_{i+1}}^{2}\right\vert A^{\phi }\left( \frac{v-t_{i}}{t_{i+1}-t_{i}%
},\frac{u-t_{i}}{t_{i+1}-t_{i}}\right) \right\vert ^{q/2}\right) ^{1/q} \\
&=&c_{\phi ,q}\left( t_{i+1}-t_{i}\right) ^{1/2}\left\vert \frac{u-v}{%
t_{i+1}-t_{i}}\right\vert \leq c_{\phi ,q}\left\vert u-v\right\vert ^{1/2}.
\end{eqnarray*}%
At last, we easily see that, for any $t_{i}\in D_{n}$, 
\begin{equation*}
\left\vert B^{n}\right\vert _{1\text{-H\"{o}l;}\left[ t_{i},t_{i+1}\right]
}\leq \left\vert \phi ^{\prime }\right\vert _{\infty }\left\vert
B^{D_{n}}\right\vert _{1\text{-H\"{o}l;}\left[ t_{i},t_{i+1}\right] }.
\end{equation*}%
This shows that all assumptions of the above theorem are satisfied and so we
have

\begin{proposition}[Rough path convergence of McShane's approximation]
For all $\alpha \in \lbrack 0,1/2)$,%
\begin{equation*}
d_{\alpha \text{-H\"{o}l}}\left( S_{2}\left( B^{n}\right) ,\exp \left(
B_{t}+A_{t}+t\mathbf{\Gamma }\right) \right) \rightarrow 0\text{ in }L^{q}%
\text{ for all }q\in \lbrack 1,\infty )
\end{equation*}%
where $A_{t}$ the usual $so\left( 2\right) $-valued L\'{e}vy's area and%
\begin{equation*}
\mathbf{\Gamma }=\left( 
\begin{array}{cc}
0 & \frac{2}{\pi }A_{0,1}^{\phi } \\ 
-\frac{2}{\pi }A_{0,1}^{\phi } & 0%
\end{array}%
\right) \in so\left( 2\right) \text{.}
\end{equation*}
\end{proposition}

We note that corollary \ref{CorMcShaneAbstract} also applies to a (minor)
variation of the McShane example given in \cite{lyons-lejay-02}.

\section{RDEs with drift: the Doss-Sussmann approach\label{secRDEwithDrift}}

\subsection{Preliminaries on ODEs and flows\label{SecControlledODE}}

Consider an ordinary differential equation (ODE), driven by a smooth $%
\mathbb{R}^{d}$-valued signal $f\left( t\right) =\left( f^{1}\left( t\right)
,\ldots ,f^{d}\left( t\right) \right) ^{T}$ on a time interval $\left[
t_{0},T\right] $ along sufficiently smooth and bounded vector fields $%
V=\left( V_{1},...,V_{d}\right) $ and a drift vector field $V_{0}$%
\begin{equation*}
\mathrm{d}y=V_{0}\left( y\right) \mathrm{d}t+V(y)\mathrm{d}f,\,\,\,y\left(
t_{0}\right) =y_{0}\in \mathbb{R}^{e}.
\end{equation*}%
We also call $U_{t\leftarrow t_{0}}^{f}\left( y_{0}\right) \equiv y_{t}$ the
associated flow. Let $J$ denote the Jacobian of $U$. It satisfies the ODE
obtained by formal differentiation w.r.t.\ $y_{0}$. More specifically, 
\begin{equation*}
a\mapsto \left\{ \frac{d}{d\varepsilon }U_{t\leftarrow t_{0}}^{f}\left(
y_{0}+\varepsilon a\right) \right\} _{\varepsilon =0}
\end{equation*}%
is a linear map from $\mathbb{R}^{e}\rightarrow \mathbb{R}^{e}$ and we let $%
J_{t\leftarrow t_{0}}^{f}\left( y_{0}\right) $ denote the corresponding $%
e\times e$ matrix. It is immediate to see that 
\begin{equation}
\frac{d}{dt}J_{t\leftarrow t_{0}}^{f}\left( y_{0}\right) =\left[ \frac{d}{dt}%
M^{f}\left( U_{t\leftarrow t_{0}}^{f}\left( y_{0}\right) ,t\right) \right]
\cdot J_{t\leftarrow t_{0}}^{f}\left( y_{0}\right)
\label{EqFirstVariational}
\end{equation}%
where $\cdot $ denotes matrix multiplication and%
\begin{equation*}
\frac{d}{dt}M^{f}\left( y,t\right) =\sum_{i=1}^{d}V_{i}^{\prime }\left(
y\right) \frac{d}{dt}f_{t}^{i}+V_{0}^{\prime }\left( y\right) .
\end{equation*}%
Note that also $J_{t_{2}\leftarrow t_{0}}^{f}=J_{t_{2}\leftarrow
t_{1}}^{f}\cdot J_{t_{1}\leftarrow t_{0}}^{f}$ and that $J_{t\leftarrow
t_{0}}^{f}$ is invertible with inverse, denoted $J_{t_{0}\leftarrow t}^{f}$,
given as the flow of (\ref{EqFirstVariational}) with $f$ replaced by $%
\overleftarrow{f^{t}}\left( .\right) =f\left( t-.\right) $, i.e.\ $%
J_{t_{0}\leftarrow t}^{f}\left( .\right) =\left( J_{t\leftarrow
t_{0}}^{f}\left( .\right) \right) ^{-1}=J_{t\leftarrow t_{0}}^{%
\overleftarrow{f^{t}}}\left( .\right) $.

Now for $V=\left( V_{1},\dots ,V_{d}\right) \in \mathrm{Lip}^{2}\left( 
\mathbb{R}
^{e}\right) $ and $x\in C^{1}\left( \left[ 0,T\right] ,%
\mathbb{R}
^{d}\right) $ ODE theory tells us that $\mathrm{d}y=V\left( y\right) \mathrm{%
d}x_{t}$ has a $C^{2}$-flow. Note that the flow 
\begin{equation*}
y_{0}\mapsto \pi _{\left( V\right) }\left( 0,y_{0},x\right) _{t}\equiv
U_{t\longleftarrow 0}^{x}\left( y_{0}\right)
\end{equation*}%
is even globally Lipschitz (thanks to the boundedness which is part of the $%
\mathrm{Lip}$-definition). The associated Jacobian $J_{t\leftarrow
0}^{x}\left( .\right) $ is itself $C^{1}$ and also globally Lipschitz in $%
y_{0}$ as well as its inverse $J_{0\longleftarrow t}^{x}\left( .\right) $.
This is more than enough to see that for $V_{0}\in \mathrm{Lip}^{1}$%
\begin{equation*}
\left( J_{0\longleftarrow t}^{x}\left( \centerdot \right) \cdot V_{0}\left(
.\right) \right) \circ \pi _{\left( V\right) }\left( 0,\centerdot ,x\right)
_{t}
\end{equation*}%
is Lipschitz (in $\centerdot $), uniformly in $t$ in $\left[ 0,T\right] $.
Obviously, the above expression, as a function of $\left( \centerdot
,t\right) $, is also continuous in $t$. It follows that%
\begin{equation*}
\dot{z}_{t}=J_{0\longleftarrow t}^{x}\left( z_{t}\right) \cdot V_{0}\left(
\pi _{\left( V\right) }\left( 0,z_{t},x_{t}\right) _{t}\right)
\end{equation*}%
has a unique solution started from $z\left( 0\right) =y_{0}$. An elementary
computation shows that $\tilde{y}\left( t\right) \equiv \pi _{\left(
V\right) }\left( 0,z_{t},x\right) _{t}$ satisfies%
\begin{equation*}
\mathrm{d}\tilde{y}=V_{0}\left( \tilde{y}\right) \mathrm{d}t+V\left( \tilde{y%
}\right) \mathrm{d}x.
\end{equation*}%
This is the Doss-Sussmann method (cf.\ \cite[p.180]{RogersWilliams:BookII})
applied in a simple ODE context.

\subsection{Doss--Sussmann method for RDEs}

We return to the discussion of section \ref{SecControlledODE} and define
solutions of RDEs with drift as uniform limits of solutions of ODEs with
drift.

\begin{definition}
Let $\mathbf{x\in }C^{\alpha \text{-H\"{o}l}}\left( \left[ 0,T\right] ,G^{%
\left[ 1/\alpha \right] }%
\mathbb{R}
^{d}\right) $. If there exists a sequence $\left( x^{n}\right) _{n\in 
\mathbb{N}
}$\ of Lipschitz paths with uniform $\alpha -$H\"{o}lder bounds converging
pointwise to $\mathbf{x}$ (that is $\mathbf{x}_{t}^{n}\equiv S_{\left[
1/\alpha \right] }\left( x^{n}\right) _{t}\rightarrow \mathbf{x}_{t}$ for
every $t\in \left[ 0,T\right] $ and $\sup_{n\in 
\mathbb{N}
}\left\Vert S_{\left[ 1/\alpha \right] }x^{n}\right\Vert _{\alpha \text{-H%
\"{o}l}}<\infty $) such that for each $n\in 
\mathbb{N}
$ the RDE (in fact ODE) with drift 
\begin{equation}
\mathrm{d}y^{n}=V_{0}\left( y^{n}\right) \mathrm{d}t+V\left( y^{n}\right) 
\mathrm{d}\mathbf{x}^{n},\text{ }y^{n}\left( 0\right) =y_{0}
\label{EqOdeDrift}
\end{equation}%
has a unique solution $y^{n}$ on $\left[ 0,T\right] $, then we call any
limit point in uniform topology of $\left\{ y^{n},n\in 
\mathbb{N}
\right\} $ a solution of the RDE with drift%
\begin{equation*}
\mathrm{d}y=V_{0}\left( y\right) \mathrm{d}t+V\left( y\right) \mathrm{d}%
\mathbf{x,}\text{ }y\left( 0\right) =y_{0}
\end{equation*}%
and we also write $y=\pi _{\left( V,V_{0}\right) }\left( 0,y_{0};\left( 
\mathbf{x},t\right) \right) $ for this solution.
\end{definition}

\begin{proposition}[Doss-Sussman for RDE]
\label{PropSolutionRDE}Assume that\renewcommand{\theenumi}{\roman{enumi}}

\begin{enumerate}
\item $\mathbf{x\in }C_{0}^{\alpha \text{-H\"{o}l}}\left( \left[ 0,T\right]
,G^{\left[ 1/\alpha \right] }\left( 
\mathbb{R}
^{d}\right) \right) $, $\alpha \in \left( 0,1/2\right) ,$

\item $V_{0}\in \mathrm{Lip}^{1}\left( 
\mathbb{R}
^{e}\right) $ and $V=\left( V_{1},\dots ,V_{d}\right) \in \mathrm{Lip}%
^{\gamma +1}\left( 
\mathbb{R}
^{e}\right) $ for a $\gamma >1/\alpha ,$

\item $y_{0}\in 
\mathbb{R}
^{e}.$
\end{enumerate}

Then there exists a unique solution $y$ to the RDE with drift%
\begin{equation}
\mathrm{d}y=V_{0}\left( y\right) \mathrm{d}t+V\left( y\right) \mathrm{d}%
\mathbf{x,}\text{ }y\left( 0\right) =y_{0}\text{.}  \label{EqDSRDEdrift}
\end{equation}%
Further, this solution is $\alpha $-H\"{o}lder continuous and given as 
\begin{eqnarray}
y\left( t\right) &=&U_{t\longleftarrow 0}^{\mathbf{x}}\left( z_{t}\right) ,
\label{EqDSFlow} \\
\text{with }\dot{z}_{t} &=&W\left( t,z_{t}\right) ,\text{ }z\left( 0\right)
=y_{0}\text{,}  \label{EqDSOde}
\end{eqnarray}%
where $W\left( t,.\right) \equiv J_{0\longleftarrow t}^{\mathbf{x;.}}\cdot
V_{0}\left( \pi _{_{\left( V\right) }}\left( 0,.;\mathbf{x}\right)
_{t}\right) $, $U_{t\longleftarrow 0}^{\mathbf{x}}\left( .\right) =\pi
_{\left( V\right) }\left( 0,.;\mathbf{x}\right) _{t}$ is the flow of 
\begin{equation*}
\mathrm{d}\tilde{y}=V\left( y\right) \mathrm{d}\mathbf{x,}\text{ }\tilde{y}%
\left( 0\right) =y_{0}
\end{equation*}%
and $J_{0\longleftarrow t}^{\mathbf{x,.}}=\left( D\pi \left( 0,.;0\right)
_{t}\right) ^{-1}$ is the inverse of the Jacobian of $U_{t\longleftarrow 0}^{%
\mathbf{x}}$.
\end{proposition}

\begin{proof}
We first show that (\ref{EqDSFlow}) has a unique solution. Existence and
uniqueness of $\pi _{\left( V\right) }\left( 0,.;\mathbf{x}\right) $ (and of
a $C^{2}$ flow) follows from standard rough path theory so this boils down
to check that the ODE \ (\ref{EqDSOde}) has a unique solution on $\left[ 0,T%
\right] $. However, the vector field $W\left( t,y\right) $ is continuous in $%
t$ and $y$ and $W\left( t,\cdot \right) $ is Lipschitz continuous in $\cdot $%
, uniformly in $t$: 
\begin{eqnarray*}
\left\vert W\left( t,y\right) -W\left( t,x\right) \right\vert &=&\left\vert
J_{0\longleftarrow t}^{\mathbf{x};y}\cdot V_{0}\left( \pi _{_{\left(
V\right) }}\left( 0,y;\mathbf{x}\right) _{t}\right) -J_{0\longleftarrow t}^{%
\mathbf{x};y}\cdot V_{0}\left( \pi _{_{\left( V\right) }}\left( 0,x;\mathbf{x%
}\right) _{t}\right) \right\vert \\
&&+\left\vert J_{0\longleftarrow t}^{\mathbf{x};y}\cdot V_{0}\left( \pi
_{_{\left( V\right) }}\left( 0,x;\mathbf{x}\right) _{t}\right)
-J_{0\longleftarrow t}^{\mathbf{x};x}\cdot V_{0}\left( \pi _{_{\left(
V\right) }}\left( 0,x;\mathbf{x}\right) _{t}\right) \right\vert \\
&\leq &\left\vert J_{0\longleftarrow t}^{\mathbf{x};y}\right\vert \left\vert
V_{0}\right\vert _{\text{$\mathrm{Lip}$}}\left\vert U_{0\rightarrow t}^{%
\mathbf{x}}\right\vert _{\text{$\mathrm{Lip}$}}\left\vert x-y\right\vert
+\left\vert J_{0\longleftarrow t}^{\mathbf{x};.}\right\vert _{\text{$\mathrm{%
Lip}$}}\left\vert x-y\right\vert \left\vert V_{0}\right\vert _{\infty }\text{%
.}
\end{eqnarray*}

Thanks to the invariance of the Lipschitz norms $\left\vert V\left(
.+y\right) \right\vert _{\text{$\mathrm{Lip}^{\gamma +1}$}}=\left\vert
V\left( .\right) \right\vert _{\text{$\mathrm{Lip}^{\gamma +1}$}}$ and
uniform estimates follow from a routine exercise showing that $\left\vert
U_{t\longleftarrow 0}^{\mathbf{x}}\right\vert _{\text{$\mathrm{Lip}$}%
},\sup_{y}\left\vert J_{0\longleftarrow t}^{\mathbf{x};y}\right\vert $ and $%
\left\vert J_{0\longleftarrow t}^{\mathbf{x}}\right\vert _{\text{$\mathrm{Lip%
}$}}$ are all bounded by a constant $c_{0}=c_{0}\left( \alpha ,\gamma
,\left\vert V\right\vert _{\mathrm{Lip}^{\gamma +1}},\left\Vert \mathbf{x}%
\right\Vert _{\alpha \text{-H\"{o}l};\left[ 0,T\right] }\right) $, uniformly
in $t$. The desired Lipschitz regularity of $W$ follows which implies
existence of a unique solution $z$ on $\left[ 0,T\right] $ of (\ref{EqDSOde}%
).

To see that the path $y_{t}=U_{t\longleftarrow 0}^{\mathbf{x}}\left(
z_{t}\right) $ is the unique RDE solution to (\ref{EqDSRDEdrift}) let $%
\left( x^{n}\right) _{n\geq 1}$ be a sequence of Lipschitz paths with
uniform $\alpha -$H\"{o}lder bounds converging pointwise to $\mathbf{x}$.
For brevity set $\mathbf{x}^{n}\equiv S_{\left[ 1/\alpha \right] }\left(
x^{n}\right) $. Following our discussion in section \ref{SecControlledODE}
the solutions $y^{n}=\pi _{\left( V,V_{0}\right) }\left( 0,y_{0};\left( 
\mathbf{x}^{n},t\right) \right) $ are given by solving (\ref{EqDSFlow}) and (%
\ref{EqDSOde}) where $\mathbf{x}$ is replaced by $\mathbf{x}^{n}$ (same
reasoning as in the first part of this proof gives existence and uniqueness
of solutions $z^{n}$).

Note that by the universal limit theorem the map $\left( y,\mathbf{x}\right)
\rightarrow \pi _{\left( V\right) }\left( 0,y,\mathbf{x}\right) $ is
uniformly continuous on bounded sets\footnote{$\left( y,\mathbf{x}\right) $
seen as an element in a product space of two metric spaces, i.e.\ with
metric $d\left( \left( y,\mathbf{x}\right) ,\left( \tilde{y},\mathbf{\tilde{x%
}}\right) \right) =\left\vert y-\tilde{y}\right\vert +d_{\alpha \text{-H\"{o}%
l}}\left( \mathbf{x,\tilde{x}}\right) $.}, so if we can show uniform
convergence of $z^{n}$ to $z$ the desired conclusion follows. A standard ODE
estimate (a simple consequence of Gronwall) is%
\begin{eqnarray*}
\sup_{t\in \left[ 0,T\right] }\left\vert z_{t}^{n}-z_{t}\right\vert &\leq &%
\frac{M^{n}}{L}\left( e^{Lt}-1\right) \\
\text{where }M^{n} &=&\sup_{\substack{ t\in \left[ 0,T\right]  \\ y\in 
\mathbb{R}
^{e}}}\left\vert W^{n}\left( t,y\right) -W\left( t,y\right) \right\vert 
\text{ and }L=\sup_{\substack{ t\in \left[ 0,T\right]  \\ x\neq y\in 
\mathbb{R}
^{e}}}\frac{\left\vert W\left( t,x\right) -W\left( t,y\right) \right\vert }{%
\left\vert x-y\right\vert }\text{.}
\end{eqnarray*}%
From the first part of this proof, $L<\infty $ and to see $%
M_{t}^{n}\rightarrow 0$ as $n\rightarrow \infty $, recall that $%
J_{0\longleftarrow t}^{\mathbf{x}^{n};y}\rightarrow J_{0\longleftarrow t}^{%
\mathbf{x};y}$ and $V_{0}\left( \pi _{_{\left( V\right) }}\left( 0,y;\mathbf{%
x}^{n}\right) _{t}\right) \rightarrow V_{0}\left( \pi _{_{\left( V\right)
}}\left( 0,y;\mathbf{x}\right) _{t}\right) $ uniformly in $y,t$ as $%
n\rightarrow \infty $ by the universal limit theorem\footnote{%
for the convergence of the Jacobian a localisation argument is actually
needed.}.

Finally, $\alpha $--H\"{o}lder continuity of the solution follows from the
estimate%
\begin{eqnarray}
\left\vert \pi _{\left( V\right) }\left( 0,z_{t};\mathbf{x}\right) _{t}-\pi
_{\left( V\right) }\left( 0,z_{s};\mathbf{x}\right) _{s}\right\vert &\leq
&\left\vert \pi _{\left( V\right) }\left( 0,z_{t};\mathbf{x}\right) _{t}-\pi
_{\left( V\right) }\left( 0,z_{t};\mathbf{x}\right) _{s}\right\vert  \notag
\\
&&+\left\vert \pi _{\left( V\right) }\left( 0,z_{t};\mathbf{x}\right)
_{s}-\pi _{\left( V\right) }\left( 0,z_{s};\mathbf{x}\right) _{s}\right\vert
\notag \\
&\leq &c_{1}\left\vert t-s\right\vert ^{\alpha }  \label{EqHoelderCont} \\
&&+\left\vert t-s\right\vert ^{\alpha }c_{2}\left\vert V\right\vert _{\text{$%
\mathrm{Lip}^{\gamma +1}$}}\left\vert z_{t}-z_{s}\right\vert
e^{c_{2}\left\Vert \mathbf{x}\right\Vert _{\alpha \text{-H\"{o}l}%
}^{a}\left\vert V\right\vert _{\text{$\mathrm{Lip}^{\gamma +1}$}}^{\alpha }}
\notag
\end{eqnarray}

where $c_{1}=c_{1}\left( \alpha ,\gamma ,\left\vert V\right\vert _{\text{$%
\mathrm{Lip}^{\gamma +1}$}}\left\Vert x\right\Vert _{\alpha \text{-H\"{o}l};%
\left[ 0,T\right] }\right) ,c_{2}=c_{2}\left( \alpha ,\gamma \right) $.
\end{proof}

\begin{remark}
The above proof can be adapted to show uniqueness and (global) existence of
RDEs with linear drift term, i.e.%
\begin{equation*}
\mathrm{d}y=Ay\mathrm{d}t+V\left( y\right) \mathrm{d}\mathbf{x,}\text{ }%
y\left( 0\right) =y_{0}\text{,}
\end{equation*}

$A$ a $\left( e\times e\right) $-matrix, same assumptions as above on $V$
and $\mathbf{x}$.
\end{remark}

\begin{corollary}
Let $\mathbf{x},V_{0}$ and $V$ be as in Proposition \ref{PropRDEwithDrift}
and let $\mathbf{x}^{n}$ be a sequence of weak geometric $\alpha $-H\"{o}%
lder paths converging with uniform bounds to $\mathbf{x}$ in supremums norm,
i.e.%
\begin{eqnarray*}
\sup_{n}\left\Vert \mathbf{x}^{n}\right\Vert _{\alpha \text{-H\"{o}l}}
&<&\infty \\
d_{\infty }\left( \mathbf{x}^{n},\mathbf{x}\right) &\rightarrow &0\text{ as }%
n\rightarrow \infty \text{.}
\end{eqnarray*}%
Denote by $y$ and $y^{n}$ the corresponding solutions of the RDE with drift (%
\ref{EqDSRDEdrift}). Then%
\begin{eqnarray*}
\sup_{n}\left\vert y^{n}\right\vert _{\alpha \text{-H\"{o}l}} &<&\infty \\
\sup_{t\in \left[ 0,T\right] }\left\vert y_{t}-y_{t}^{n}\right\vert
&\rightarrow &0\text{ as }n\rightarrow \infty \text{.}
\end{eqnarray*}%
and by interpolation for every $\alpha ^{\prime }<\alpha $%
\begin{equation*}
\left\vert y-y^{n}\right\vert _{\alpha ^{\prime }\text{-H\"{o}l}}\rightarrow
0\text{ as }n\rightarrow \infty \text{.}
\end{equation*}
\end{corollary}

\begin{proof}
The uniform H\"{o}lder bounds follow as in (\ref{EqHoelderCont}) by noting
that the $z^{n}$ are uniformly bounded. For every $\mathbf{x}^{n}$ there
exists a sequence of Lipschitz paths $\left( x^{n,m}\right) _{m}$ with $S_{%
\left[ 1/\alpha \right] }\left( x^{n,m}\right) $ converging to $\mathbf{x}%
^{n}$ in $d_{\infty }$ with \ uniform (in $m$) H\"{o}lder bounds and with an
associated sequence of RDE solutions $y^{n,m}$ converging to $y^{n}$. One
can choose a subsequence $m_{n}$ such that $S_{\left[ 1/\alpha \right]
}\left( x^{n,m_{n}}\right) $ converges to $\mathbf{x}$ with uniform H\"{o}%
lder bounds and $\sup_{t}\left\vert y_{t}^{n,m_{n}}-y_{t}^{n}\right\vert
\rightarrow 0$. Hence, 
\begin{equation*}
\sup_{t}\left\vert y_{t}-y_{t}^{n}\right\vert \leq \sup_{t}\left\vert
y_{t}-y_{t}^{n,m_{n}}\right\vert +\sup_{t}\left\vert
y_{t}^{n,m_{n}}-y_{t}^{n}\right\vert \rightarrow 0\text{ as }n\rightarrow
\infty \text{.}
\end{equation*}
\end{proof}

\begin{remark}
The auxiliary differential equation for $z$ can be written as%
\begin{equation}
\mathrm{d}z=\left( J_{0\longleftarrow t}^{\mathbf{x}}.W\right) \circ \pi
\left( 0,z_{t},\mathbf{x}\right) _{t}\mathrm{d}h  \label{EqOdeZYoung}
\end{equation}%
where $W=V_{0}$ and $h\left( t\right) =t$. In fact, one can take $W=\left(
W_{1},\dots ,W_{d^{\prime }}\right) $, $h\in C^{1\text{-var}}\left( \left[
0,T\right] ,\mathbb{R}^{d^{\prime }}\right) $ in which case $\pi _{\left(
V\right) }\left( 0,z_{t},\mathbf{x}\right) _{t}$ solves%
\begin{equation*}
\mathrm{d}\tilde{y}=V\left( \tilde{y}\right) \mathrm{d}\mathbf{x}+W\left( 
\tilde{y}\right) \mathrm{d}h.
\end{equation*}%
We could make sense of (\ref{EqOdeZYoung}) as Young-integral equation as
long as $1/p+1/q>1$ and thus obtain RDEs with drift-vector fields driven by $%
h$. In this case the pair $\left( \mathbf{x},h\right) $ gives rise to a
rough path (the cross-integrals of $\mathbf{x}$ and $h$ are well-defined
Young-integrals); the advantage of the present consideration would be to
reduce the regularity assumptions on $W$.
\end{remark}

\subsection{Euler scheme for RDEs with drift}

We recall the Euler scheme for RDEs, cf. \cite%
{FrizVictoir06:EulerEstimatesforRDEs}.

\begin{definition}
Let $N\in 
\mathbb{N}
$. Given $\mathrm{Lip}^{N}$ vector fields $V=\left( V_{i}\right)
_{i=1,\ldots ,d}$ on $%
\mathbb{R}
^{e},$ $\mathbf{g}\in G^{N}\left( 
\mathbb{R}
^{e}\right) ,$ $y\in 
\mathbb{R}
^{e}$. We call%
\begin{equation*}
\mathcal{E}_{\left( V\right) }^{N}\left( y,\mathbf{g}\right)
:=\sum_{k=1}^{N}\sum_{\substack{ i_{1},\ldots ,i_{k}\in  \\ \left\{ 1,\ldots
,d\right\} }}V_{i_{1}}\cdots V_{i_{k}}I\left( y\right) \mathbf{g}%
^{k;i_{1},\ldots ,i_{k}}
\end{equation*}%
the step-$N$ Euler scheme ($I$ denotes the identity map).
\end{definition}

\begin{proposition}[Euler-estimate for RDEs with drift]
\label{PropEuler}Let $N\in 
\mathbb{N}
,$ $N\geq 1/\alpha $. For $\mathbf{x\in }C_{0}^{\alpha \text{-H\"{o}l}%
}\left( \left[ 0,T\right] ,G^{N}\left( 
\mathbb{R}
^{d}\right) \right) $, $V_{0}\in \mathrm{Lip}^{1},$ $V=\left( V_{i}\right)
_{i=1,\ldots ,d}\in \mathrm{Lip}^{\gamma +1}\left( 
\mathbb{R}
^{e}\right) ,\gamma >N$%
\begin{equation*}
\left\vert \pi _{\left( V_{0},V\right) }\left( s,y_{s};\left( \mathbf{x}%
,t\right) \right) _{s,t}-\mathcal{E}_{\left( V\right) }^{N}\left( y_{s},%
\mathbf{x}_{s,t}\right) -V_{0}\left( y_{s}\right) \left\vert t-s\right\vert
\right\vert \leq c\left\vert t-s\right\vert ^{\theta }
\end{equation*}%
where $\theta \geq 1+\alpha >1$ and $c=c\left( \alpha ,N,\left\vert
y_{s}\right\vert ,\left\Vert \mathbf{x}\right\Vert _{\alpha \text{-H\"{o}l}%
},\left\vert V\right\vert _{\mathrm{Lip}^{N}},\left\vert V_{0}\right\vert _{%
\mathrm{Lip}^{1}}\right) $. Here $y_{s}\in 
\mathbb{R}
^{e}$ is a fixed "starting" point.
\end{proposition}

\begin{proof}
This is an error estimate for RDEs with drift over the time-interval $\left[
s,t\right] $. By shifting time, we may consider w.l.o.g. $s=0$, and from 
\cite{FrizVictoir06:EulerEstimatesforRDEs} we know that%
\begin{equation*}
\left\vert \pi _{\left( V\right) }\left( 0,y_{0},\mathbf{x}\right) _{0,t}-%
\mathcal{E}_{\left( V\right) }^{N}\left( y_{0},\mathbf{x}_{0,t}\right)
\right\vert \leq c_{0}t^{\theta }
\end{equation*}%
where $c_{0}=c_{0}\left( \alpha ,N,y_{0},\left\Vert \mathbf{x}\right\Vert
_{\alpha \text{-H\"{o}l}},\left\vert V\right\vert _{\text{Lip}^{N}}\right) $%
, $\theta $ $=\left( N+1\right) \alpha \geq 1+\alpha $. By the triangle
inequality and our definition of "RDE with drift" it then suffices to show
that%
\begin{equation*}
\left\vert \pi _{\left( V_{0},V\right) }\left( 0,y_{0};\left( \mathbf{x}%
,t\right) \right) _{0,t}-\pi _{\left( V\right) }\left( 0,y_{0},\mathbf{x}%
\right) _{0,t}-V_{0}\left( y_{0}\right) t\right\vert \leq c_{1}t^{\theta }
\end{equation*}%
To see this, recall%
\begin{equation*}
\pi _{\left( V_{0},V\right) }\left( 0,y_{0};\left( \mathbf{x},t\right)
\right) _{0,t}=\pi _{\left( V\right) }\left( 0,z_{t},\mathbf{x}\right) _{t}
\end{equation*}%
where $z_{t}-z_{0}=z_{t}-y_{0}=$ $V_{0}\left( y_{0}\right) t+O\left(
t^{2}\right) $. We then have%
\begin{eqnarray*}
\pi _{\left( V\right) }\left( 0,z_{t},\mathbf{x}\right) _{t}-\pi _{\left(
V\right) }\left( 0,y_{0},\mathbf{x}\right) _{t}
&=&\int_{0}^{t}J_{t\longleftarrow 0}^{\mathbf{x}}|_{\pi _{\left( V\right)
}\left( 0,z_{s},\mathbf{x}\right) _{0,t}}\mathrm{d}z_{s} \\
&=&\int_{0}^{t}\mathrm{d}z+\int_{0}^{t}\left( J_{t\longleftarrow 0}^{\mathbf{%
x}}|_{\pi _{\left( V\right) }\left( 0,z_{s},\mathbf{x}\right)
_{0,t}}-I\right) \mathrm{d}z_{s} \\
&=&V_{0}\left( y_{0}\right) t+t^{1+\alpha }
\end{eqnarray*}%
and the proof is finished.
\end{proof}

\section{Applications}

\subsection{Drift vector fields induced by perturbed driving signals}

We now show that perturbations of a rough driving signal are picked up by
the RDE as a drift term of iterated Lie brackets of the vector fields. Since
the RDE solution is a continuous function of the driving signal, we also
have continuity under convergence in probability (in suitable H\"{o}lder
rough path metrics) of random rough driving signals. Combined with the
results of section \ref{secApprox} we arrive at a general criterion for
non-standard convergence in stochastic differential equations, more general
than \cite{ikeda-watanabe-89, Gyongy:WongZakaiMartingales} in the sense that
our result applies to perturbations on\textit{\ all levels}, exhibiting
additional drift terms involving any iterated Lie brackets of the driving
vector fields. In particular, the examples given in section \ref{secApprox}
allows us to recover the convergence results of McShane \cite%
{McShane72:SDEandModels} and Sussmann \cite{Sussmann91:WongZakaiWithDrift}.
(In fact, a free benefit of the rough path approach, the respective
convergence results will take place at the level of stochastic flows.)

\begin{theorem}
\label{PropRDEwithDrift}\bigskip Let $\mathbf{x}:\left[ 0,T\right]
\rightarrow G^{\left[ 1/\alpha \right] }\left( 
\mathbb{R}
^{d}\right) $ be a weak geometric $\alpha $-H\"{o}lder rough path, fix $%
\mathbf{v=}\left( v^{i_{1},\ldots ,i_{N}}\right) _{i_{1},\ldots
,i_{N}=1,\ldots ,d}\mathbf{\in }\mathcal{V}^{N}\left( 
\mathbb{R}
^{d}\right) ,$ $N\geq 2$ and set%
\begin{equation*}
\mathbf{\tilde{x}}_{s,t}=\exp \left( \log \mathbf{x}_{s,t}+\mathbf{v}%
t\right) \text{ \ for }s,t
\end{equation*}%
(This defines a weak geometric $min\left( \alpha ,1/N\right) $-H\"{o}lder
rough path in $G^{N}\left( 
\mathbb{R}
^{d}\right) $ with identical projection as $\mathbf{x}$ to $%
\mathbb{R}
^{d}$). Further, assume $V_{0}\in \mathrm{Lip}^{1},V=\left( V_{i}\right)
_{i=1,\ldots ,d}\in \mathrm{Lip}^{\gamma +1},\gamma >\max \left( 1/\alpha
,N\right) ,$ vector fields on $%
\mathbb{R}
^{e}$. Then the unique RDE solution of%
\begin{equation}
\mathrm{d}y=V_{0}\left( y\right) \mathrm{d}t+V\left( y\right) \mathrm{d}%
\mathbf{\tilde{x}},\text{ }y\left( 0\right) =y_{0}
\label{EqRDEperturbedPath}
\end{equation}%
coincides with the unique RDE solution of%
\begin{equation}
\mathrm{d}z=\left( V_{0}\left( z\right) +W\left( z\right) \right) \mathrm{d}%
t+V\left( z\right) \mathrm{d}\mathbf{x},\text{ }z\left( 0\right) =y_{0}
\label{EqRDEdrift}
\end{equation}%
where 
\begin{equation*}
W\left( z\right) \equiv \sum_{i_{1},\ldots ,i_{\left[ 1/\alpha \right] }\in
\left\{ 1,\ldots ,d\right\} }\left. \left[ V_{i_{1}},\left[ \ldots ,\left[
V_{i_{\left[ 1/\alpha \right] -1}},V_{i_{\left[ 1/\alpha \right] }}\right] %
\right] \ldots \right] \right\vert _{z}\mathbf{v}^{i_{1},\ldots ,i_{\left[
1/\alpha \right] }}
\end{equation*}
\end{theorem}

We prepare the proof with

\begin{lemma}
\label{LemmaHoermander}\bigskip Let $k\in 
\mathbb{N}
$. Given a multi-index $\alpha =\left( \alpha _{1},...,\alpha _{k}\right)
\in \left\{ 1,...,d\right\} ^{k}$ and $k$ $\mathrm{Lip}^{k}$ vector fields $%
V_{1},\dots ,V_{k}$ on $\mathbb{R}^{e}$, define 
\begin{equation*}
V_{\alpha }=\left[ V_{\alpha _{k}},\left[ V_{\alpha _{k-1}},...,\left[
V_{\alpha _{2}},V_{\alpha _{1}}\right] \right] \right] .
\end{equation*}%
Further let $e_{1},...,e_{d}$ denote the canonical basis of $\mathbb{R}^{d}$%
. Then $\mathfrak{g}^{n}\left( \mathbb{R}^{d}\right) $, the step-$n$ free
Lie algebra, is generated by elements of form%
\begin{equation*}
e_{\alpha }=\left[ e_{\alpha _{k}},\left[ e_{\alpha _{k-1}},...,\left[
e_{\alpha _{2}},e_{\alpha _{1}}\right] \right] \right] \in \left( \mathbb{R}%
^{d}\right) ^{\otimes k},\,\,\,k\leq n
\end{equation*}%
with $\left[ e_{i},e_{j}\right] =$ $e_{i}\otimes e_{j}-e_{j}\otimes e_{i}$
and%
\begin{equation*}
\sum_{i_{1},...,i_{k}\in \left\{ 1,...,d\right\} }V_{i_{k}}\cdots
V_{i_{1}}\left( e_{\alpha }\right) ^{i_{k},...,i_{1}}=V_{\alpha }.
\end{equation*}
\end{lemma}

\begin{proof}
It is clear that $\mathfrak{g}^{n}\left( \mathbb{R}^{d}\right) $ is
generated by the $e_{\alpha }$. We prove the second statement by induction:
a straightforward calculation shows that it holds for $k=2$. Now suppose it
holds for $k-1$ and denote $V_{\tilde{\alpha}}=\left[ V_{\alpha _{k-1}},...%
\left[ V_{\alpha _{2}},V_{\alpha _{1}}\right] \right] $. Then (using
summation convention)\allowdisplaybreaks%
\begin{eqnarray*}
V_{i_{k}}\dots V_{i_{1}}\left( e_{\alpha }\right) ^{i_{k},...,i_{1}}
&=&V_{i_{k}}\dots V_{i_{1}}\left( e_{\alpha _{k}}\otimes \left[ e_{\alpha
_{k-1}},...,\left[ e_{\alpha _{2}},e_{\alpha _{1}}\right] \right] \right)
^{i_{k},...,i_{1}} \\
&&-V_{i_{k}}\dots V_{i_{1}}\left( \left[ e_{\alpha _{k-1}},...,\left[
e_{\alpha _{2}},e_{\alpha _{1}}\right] \right] \otimes e_{\alpha
_{k}}\right) ^{i_{k},...,i_{1}} \\
&=&V_{i_{k}}\dots V_{i_{1}}\delta ^{\alpha _{k},i_{k}}\otimes \left[
e_{\alpha _{k-1}},...,\left[ e_{\alpha _{2}},e_{\alpha _{1}}\right] \right]
^{i_{k-1},...,i_{1}} \\
&&-V_{i_{k}}\dots V_{i_{1}}\left[ e_{\alpha _{k-1}},...,\left[ e_{\alpha
_{2}},e_{\alpha _{1}}\right] \right] ^{i_{k},...,i_{2}}\otimes \delta
^{\alpha _{k},i_{1}} \\
&=&V_{\alpha _{k}}V_{i_{k-1}}\dots V_{i_{1}}\left[ e_{\alpha _{k-1}},...,%
\left[ e_{\alpha _{2}},e_{\alpha _{1}}\right] \right] ^{i_{k-1},...,i_{1}} \\
&&-V_{i_{k}}\dots V_{i_{2}}\left[ e_{\alpha _{k-1}},...,\left[ e_{\alpha
_{2}},e_{\alpha _{1}}\right] \right] ^{i_{k},...,i_{2}}V_{\alpha _{k}} \\
&=&V_{\alpha _{k}}V_{\tilde{\alpha}}-V_{\tilde{\alpha}}V_{\alpha _{k}}=\left[
V_{\alpha _{k}},\left[ V_{\alpha _{k-1}},...,\left[ V_{\alpha
_{2}},V_{\alpha _{1}}\right] \right] \right]
\end{eqnarray*}%
(where we used the induction hypothesis that 
\begin{equation*}
V_{i_{k-1}}\dots V_{i_{1}}\left[ e_{\alpha _{k-1}},...,\left[ e_{\alpha
_{2}},e_{\alpha _{1}}\right] \right] ^{i_{k-1},...,i_{1}}=V_{i_{k}}\dots
V_{i_{2}}\left[ e_{\alpha _{k-1}},...,\left[ e_{\alpha _{2}},e_{\alpha _{1}}%
\right] \right] ^{i_{k},...,i_{2}}=V_{_{\tilde{\alpha}}}\text{).}
\end{equation*}
\end{proof}

\begin{proof}[Proof of Theorem \protect\ref{PropRDEwithDrift}]
By construction $W\in \mathrm{Lip}^{1}$ and existence and uniqueness of RDE
solutions $y,z$ to (\ref{EqRDEdrift}), (\ref{EqRDEperturbedPath}) follows
from proposition \ref{PropSolutionRDE}. We have to show that $y_{t}=z_{t}$
for every $t\in \left[ 0,T\right] $. Therefore fix $\tilde{T}\in \left[ 0,T%
\right] $, take a dissection $D=\left( t_{i}\right) _{i=0,...,\left\vert
D\right\vert }$ of $\left[ 0,\tilde{T}\right] $ with $t_{0}=0$ and $%
t_{\left\vert D\right\vert }=\tilde{T}$ and define%
\begin{equation*}
z_{s}^{i}=\pi _{\left( V_{0}+W,V\right) }\left( t_{i},y_{t_{i}};\left( 
\mathbf{x,}t\right) \right) _{s}\text{ for }s\in \left[ t_{i},\tilde{T}%
\right] \subset \left[ 0,1\right] ,\text{ }i=1,\ldots ,d.
\end{equation*}%
Note that $z_{\tilde{T}}^{0}=z_{\tilde{T}}$ and $z_{\tilde{T}}^{\left\vert
D\right\vert }=y_{\tilde{T}}$, hence%
\begin{equation}
\left\vert z_{\tilde{T}}-y_{\tilde{T}}\right\vert =\left\vert z_{\tilde{T}%
}^{\left\vert D\right\vert }-z_{\tilde{T}}^{0}\right\vert \leq
\sum_{i=1}^{\left\vert D\right\vert }\left\vert z_{\tilde{T}}^{i}-z_{\tilde{T%
}}^{i-1}\right\vert .  \label{EqTriangleIneq}
\end{equation}%
Using the Lipschitz continuity of RDE flows we get\allowdisplaybreaks%
\begin{eqnarray}
\left\vert z_{\tilde{T}}^{i}-z_{\tilde{T}}^{i-1}\right\vert &=&\left\vert
\pi _{\left( V_{0}+W,V\right) }\left( t_{i},y_{t_{i}};\left( \mathbf{x,}%
t\right) \right) _{\tilde{T}}-\pi _{\left( V_{0}+W,V\right) }\left(
t_{i-1},y_{t_{i-1}};\left( \mathbf{x,}t\right) \right) _{\tilde{T}%
}\right\vert  \notag \\
&=&\left\vert \pi _{\left( V_{0}+W,V\right) }\left( t_{i},y_{t_{i}};\left( 
\mathbf{x,}t\right) \right) _{\tilde{T}}-\pi _{\left( V_{0}+W,V\right)
}\left( t_{i},\pi _{\left( V_{0}+W,V\right) }\left(
t_{i-1},y_{t_{i-1}};\left( \mathbf{x,}t\right) \right) _{t_{i}};\left( 
\mathbf{x,}t\right) \right) _{\tilde{T}}\right\vert  \notag \\
&\leq &\left\vert U_{\tilde{T}\longleftarrow t_{i}}^{\left( \mathbf{x}%
,t\right) }\right\vert _{\text{$\mathrm{Lip}$}}\left\vert
y_{t_{i-1},t_{i}}-\pi _{\left( V_{0}+W,V\right) }\left(
t_{i-1},y_{t_{i-1}};\left( \mathbf{x,}t\right) \right)
_{t_{i-1},t_{i}}\right\vert  \notag
\end{eqnarray}%
For brevity set $\alpha ^{\ast }=\min \left( \alpha ,1/N\right) .$ By
adding/substracting $\mathcal{E}_{\left( V\right) }^{\left[ 1/\alpha ^{\ast }%
\right] }\left( y_{t_{i-1}},\mathbf{\tilde{x}}_{t_{i-1},t_{i}}\right)
+V_{0}\left( y_{t_{i-1}}\right) \left\vert t_{i}-t_{i-1}\right\vert $ and
splitting up we estimate $\left\vert z_{\tilde{T}}^{i}-z_{\tilde{T}%
}^{i-1}\right\vert \leq \left( 1\right) +\left( 2\right) $ with 
\begin{eqnarray*}
\left( 1\right) &=&\left\vert \pi _{\left( V_{0},V\right) }\left(
t_{i-1},y_{t_{i-1}};\left( \mathbf{\tilde{x},}t\right) \right)
_{t_{i-1},t_{i}}-\mathcal{E}_{\left( V\right) }^{\left[ 1/\alpha ^{\ast }%
\right] }\left( y_{t_{i-1}},\mathbf{\tilde{x}}_{t_{i-1},t_{i}}\right)
-V_{0}\left( y_{t_{i-1}}\right) \left\vert t_{i}-t_{i-1}\right\vert
\right\vert \\
\left( 2\right) &=&\left\vert \mathcal{E}_{\left( V\right) }^{\left[
1/\alpha ^{\ast }\right] }\left( y_{t_{i-1}},\mathbf{\tilde{x}}%
_{t_{i-1},t_{i}}\right) +V_{0}\left( y_{t_{i-1}}\right) \left\vert
t_{i}-t_{i-1}\right\vert -\pi _{\left( V_{0}+W,V\right) }\left(
t_{i-1},y_{t_{i-1}};\left( \mathbf{x,}t\right) \right)
_{t_{i-1},t_{i}}\right\vert \text{.}
\end{eqnarray*}%
From proposition \ref{PropEuler}, $\left( 1\right) \leq c_{1}\left\vert
t-s\right\vert ^{\theta }$, $\theta >1$, and by lemma \ref{LemmaHoermander},%
\begin{eqnarray*}
\mathcal{E}_{\left( V\right) }^{\left[ 1/\alpha ^{\ast }\right] }\left(
y_{t_{i-1}};\mathbf{\tilde{x}}_{t_{i-1},t_{i}}\right) &=&\mathcal{E}_{\left(
V\right) }^{\left[ 1/\alpha ^{\ast }\right] }\left( y_{t_{i-1}};\mathbf{x}%
_{t_{i-1},t_{i}}\right) +\sum_{i_{1},\ldots ,i_{N}}V_{i_{1}}\cdots
V_{i_{N}}I\left( y_{t_{i-1}}\right) v^{i_{1},\ldots ,i_{N}}\left\vert
t_{i}-t_{i-1}\right\vert \\
&=&\mathcal{E}_{\left( V\right) }^{\left[ 1/\alpha ^{\ast }\right] }\left(
y_{t_{i-1}};\mathbf{x}_{t_{i-1},t_{i}}\right) +\left\vert
t_{i}-t_{i-1}\right\vert W\left( y_{t_{i-1}}\right) \text{.}
\end{eqnarray*}%
Again proposition \ref{PropEuler} applies and $\left( 2\right) \leq
c_{2}\left\vert t-s\right\vert ^{\theta }$. Plugging all this into (\ref%
{EqTriangleIneq}) gives%
\begin{equation*}
\left\vert z_{\tilde{T}}-y_{\tilde{T}}\right\vert \leq
c_{3}\sum_{i=1}^{\left\vert D\right\vert }\left\vert
t_{i}-t_{i-1}\right\vert ^{\theta }
\end{equation*}%
with $c_{3}=c_{3}\left( \alpha ,N,\left\Vert \mathbf{\tilde{x}}\right\Vert
_{\alpha ^{\ast }\text{-H\"{o}l}},\left\Vert \mathbf{x}\right\Vert _{\alpha 
\text{-H\"{o}l}},\left\vert V\right\vert _{\mathrm{Lip}^{\left[ 1/\alpha
^{\ast }\right] }},\left\vert V_{0}\right\vert _{\mathrm{Lip}%
^{1}},\left\vert W\right\vert _{\mathrm{Lip}^{1}}\right) $. Since $\theta >1$
the sum on the r.h.s goes to $0$ as $\left\vert D\right\vert \rightarrow 0$
and this finishes the proof.
\end{proof}

\subsection{Optimality of RDE estimates\label{SectOpt}}

At last, we use theorem \ref{PropRDEwithDrift} to establish optimality of
two important RDE estimates (proved in \cite{lyons-98} and \cite%
{frizvictoirbook}) for the case of paths with $p$-variation, $p\in 
\mathbb{N}
$. The second part of the following theorem settles a questions that was eft
open in \cite{CassFrizVictoir:WienerFunct}.

\begin{theorem}
\label{ThLinearVf}Let $\mathbf{x}:\left[ 0,T\right] \rightarrow G^{p}\left( 
\mathbb{R}
^{d}\right) $ be a geometric $p$-rough path, $y_{0}\in 
\mathbb{R}
^{e},p\in 
\mathbb{N}
$. If either

\begin{enumerate}
\item $V=\left( V_{i}\right) _{i=1}^{d}\in \mathrm{Lip}^{\gamma }\left( 
\mathbb{R}
^{e}\right) ,\gamma >p$ or

\item $V=\left( V_{i}\right) _{i=1}^{d}$ linear vector fields on $%
\mathbb{R}
^{d}$, i.e.\ $V_{i}\left( y\right) =A_{i}\cdot y$ with $A_{i}$ a $\left(
e\times e\right) $-matrix
\end{enumerate}

then there exists a unique RDE solution to%
\begin{equation*}
\mathrm{d}y=V\left( y\right) \mathrm{d}\mathbf{x,}\text{ }y\left( 0\right)
=y_{0}\text{.}
\end{equation*}%
Further, in case ($1$),%
\begin{equation}
\left\vert y\right\vert _{\infty }\leq C\max (\left\Vert \mathbf{x}%
\right\Vert _{p\text{-var;}\left[ s,t\right] },\left\Vert \mathbf{x}%
\right\Vert _{p\text{-var;}\left[ s,t\right] }^{p})
\label{EqEstimateBoundedVf}
\end{equation}%
with $C=C\left( y_{0},p,\left\vert V\right\vert _{\mathrm{Lip}^{\gamma
+1}}\right) $ and in case ($2$) 
\begin{equation}
\left\vert y\right\vert _{\infty }\leq c\exp \left( c\left\Vert \mathbf{x}%
\right\Vert _{p\text{-var}}^{p}\right)  \label{EqLinearVfEstimate}
\end{equation}%
with $c=c\left( y_{0},p,\left( \left\vert V_{i}\right\vert \right)
_{i=1}^{d}\right) $ hold and both estimates are optimal (i.e.\ the\ bounds
are attained).
\end{theorem}

We prepare the proof with two lemmas

\begin{lemma}
\label{LemBoundedVfCommutator}For all $k\in 
\mathbb{N}
$ there exist $k$ smooth, bounded vector fields $V_{1},\ldots ,V_{k}$ on $%
\mathbb{R}
^{e}$ such that%
\begin{equation*}
\underbrace{\left[ V_{k},\ldots ,\left[ V_{3},\left[ V_{2},V_{1}\right] %
\right] \ldots \right] }_{k-1\text{ brackets}}=\frac{\partial }{\partial
x_{e}}
\end{equation*}
\end{lemma}

\begin{proof}
Set 
\begin{equation*}
V=\sin x_{e}\frac{\partial }{\partial x_{e}},W=-\cos x_{e}\frac{\partial }{%
\partial x_{e}},E=\frac{\partial }{\partial x_{e}}\text{.}
\end{equation*}%
Note that $\left[ V,W\right] =E$ and $\left[ V,E\right] =W$. Hence,%
\begin{eqnarray*}
\text{for }k\text{ odd}\text{: } &&\left[ V,\ldots ,\left[ V,\left[ V,W%
\right] \right] \ldots \right] =E \\
\text{for }k\text{ even} &\text{:}&\left[ V,\ldots ,\left[ V,\left[ V,E%
\right] \right] \ldots \right] =E.
\end{eqnarray*}
\end{proof}

\begin{lemma}
\label{LemMatrixCommutator}For all $k\in 
\mathbb{N}
$ there exist $\left( e\times e\right) $-matrices $A_{1},\ldots ,A_{k}$ ($%
e\geq 2$) such that%
\begin{equation*}
\underbrace{\left( \left[ A_{k},\ldots ,\left[ A_{3},\left[ A_{2},A_{1}%
\right] \right] \ldots \right] \right) _{i,j=1,\ldots ,e}}_{k-1\text{
brackets}}=\left( -\delta _{\left( i,j\right) =\left( 1,1\right) }+\delta
_{\left( i,j\right) =\left( e,e\right) }\right) _{i,j=1,\ldots ,e}
\end{equation*}%
where $\left[ .,.\right] $ denotes the usual matrix commutator.
\end{lemma}

\begin{proof}
\bigskip Let%
\begin{eqnarray*}
M &=&\left( \delta _{\left( i,j\right) =\left( 1,e\right) }\right)
_{i,j=1,\ldots ,e},\text{ }N=\left( -\delta _{\left( i,j\right) =\left(
1,1\right) }+\delta _{\left( i,j\right) =\left( e,e\right) }\right)
_{i,j=1,\ldots ,e}, \\
A &=&\left( \delta _{\left( i,j\right) =\left( e,1\right) }\right)
_{i,j=1,\ldots ,e},\text{ }B=\left( \frac{1}{2}\delta _{\left( i,j\right)
=\left( 1,e\right) }\right) _{i,j=1,\ldots ,e}.
\end{eqnarray*}%
Note that $\left[ A,M\right] =N$ and $\left[ B,N\right] =M$. Hence 
\begin{eqnarray*}
\text{for }k\text{ odd} &\text{: }&\left[ A,\left[ B,\left[ A,\ldots ,\left[
A,\left[ B,\left[ A,M\right] \right] \right] \ldots \right] \right] \right]
=N \\
\text{for }k\text{ even} &\text{: }&\left[ A,\left[ B,\left[ A,\ldots ,\left[
A,\left[ B,N\right] \right] \ldots \right] \right] \right] =N.
\end{eqnarray*}%
\bigskip
\end{proof}

\begin{proof}[Proof of Theorem \protect\ref{ThLinearVf}]
Existence of a unique RDE solution follows from \cite{lyons-98}. For (\ref%
{EqLinearVfEstimate}) use \cite[Theorem 2.4.1]{lyons-98} with control
function $\omega \left( s,t\right) =\left\Vert \mathbf{x}\right\Vert _{p%
\text{-var;}\left[ s,t\right] }^{p}$ to get%
\begin{equation*}
\left\vert y_{t}\right\vert \leq K\left\vert y_{0}\right\vert \left\Vert 
\mathbf{x}\right\Vert _{p\text{-var}}\sum_{n\geq 0}K^{n}\frac{\left\Vert 
\mathbf{x}\right\Vert _{p\text{-var}}^{n}}{\left( n/p\right) !}
\end{equation*}%
where $K=\max_{i}\left( \left\vert A_{i}\right\vert \right) $. However, $%
\sum_{n\geq 0}\frac{a^{n/p}}{\left( n/p\right) !}\leq c_{0}\left( p\right)
e^{c_{0}\left( p\right) a}$ for $a\geq 0$ and therefore%
\begin{equation*}
\left\vert y\right\vert _{\infty ,\left[ 0,T\right] }\leq \left\vert
y_{0}\right\vert c_{1}\exp \left( c_{1}\left\Vert \mathbf{x}\right\Vert _{p-%
\text{var}}^{p}\right) \text{.}
\end{equation*}%
Estimate (\ref{EqEstimateBoundedVf}) is proven in \cite{frizvictoirbook}.

To see optimality of both (\ref{EqEstimateBoundedVf}) and (\ref%
{EqLinearVfEstimate}) define a geometric $1/p$-H\"{o}lder rough path $%
\mathbf{\tilde{x}}$ on $\left[ 0,1\right] $ in $G^{p}\left( \mathbb{R}%
^{p}\right) $ by 
\begin{equation*}
\mathbf{\tilde{x}}_{t}=\exp \left( \lambda te_{1,\ldots ,p}\right)
\end{equation*}%
where $e_{1,\ldots ,p}=\left[ e_{p},\ldots ,\left[ e_{3},\left[ e_{2},e_{1}%
\right] \right] \ldots \right] $ $\in \mathcal{V}^{p}\left( 
\mathbb{R}
^{p}\right) $ for some $\lambda >0$. Note that homogeneity of the $p$%
-variation norm implies 
\begin{equation}
\left\Vert \mathbf{\tilde{x}}\right\Vert _{p\text{-var}}=c_{2}\lambda ^{1/p}%
\text{.}  \label{EqXpvar}
\end{equation}%
The corresponding RDE solution%
\begin{equation*}
\mathrm{d}y=V\left( y\right) \mathrm{d}\mathbf{\tilde{x},}\text{ }y\left(
0\right) =y_{0}\text{.}
\end{equation*}%
then coincides with the ODE solution%
\begin{equation}
\mathrm{d}z=W\left( z\right) \mathrm{d}t,\text{ }z\left( 0\right) =y_{0}
\label{EqRdeWdrift}
\end{equation}%
where $W$ is given by theorem \ref{PropRDEwithDrift}.\newline

\underline{Case $\left( 1\right) $:} .Let $V=\left( V_{i}\right)
_{i=1,\ldots ,p}$ be the vector fields on $%
\mathbb{R}
^{p}$ from lemma \ref{LemBoundedVfCommutator} with $k=p.$ Then the solution (%
\ref{EqRdeWdrift}) is easy to write down,%
\begin{equation*}
y_{t}=\left( 0,\ldots ,0,\lambda t\right) ^{T}\text{ .}
\end{equation*}%
Clearly, $\left\vert y\right\vert _{\infty }=1/c_{2}^{p}\left\Vert \mathbf{%
\tilde{x}}\right\Vert _{p\text{-var}}^{p}$. To see that the regime of (\ref%
{EqEstimateBoundedVf}) where $\left\Vert \mathbf{x}\right\Vert _{p\text{-var}%
}$ dominates can be obtained is straightforward by looking at $\mathrm{d}%
y=\left( 1,\ldots ,1\right) ^{T}\mathrm{d}\mathbf{x}$ for any rough path $%
\mathbf{x}$.\newline

\underline{Case $\left( 2\right) $:} Let $V=\left( V_{i}\right) _{i=1,\ldots
,p}$ be the linear vector fields on $%
\mathbb{R}
^{p}$ given by the matrices $\left( A_{i}\right) _{i=1,\ldots ,p}$ of lemma %
\ref{LemMatrixCommutator}, i.e.\ $V_{i}\left( y\right) =A_{i}\cdot y$. Using
lemma \ref{LemMatrixCommutator} with $k=p$ brackets\footnote{%
Using the usual identification of linear maps with matrices it is easy check
that we have a Lie algebra isomorphism $\left( \mathfrak{gl}\left( n,%
\mathbb{R}
\right) ,\left[ .,.\right] _{M}\right) \cong \left( \mathfrak{gl}\left( 
\mathbb{R}
^{n}\right) ,\left[ .,.\right] \right) $.,},%
\begin{eqnarray*}
W\left( z\right) &=&\lambda \left[ V_{p},\ldots ,\left[ V_{3},\left[
V_{2},V_{1}\right] \right] \ldots \right] I\left( z\right) \\
&=&\lambda \left[ A_{p},\ldots ,\left[ A_{3},\left[ A_{2},A_{1}\right] %
\right] \ldots \right] _{M}\cdot z \\
&=&\lambda \left( -\delta _{\left( i,j\right) =\left( 1,1\right) }+\delta
_{\left( i,j\right) =\left( p,p\right) }\right) _{i,j=1,\ldots ,e}\cdot z
\end{eqnarray*}%
Now choosing $y_{0}=\left( 0,\ldots ,0,1\right) ^{T}$, the unique solution
to (\ref{EqRdeWdrift}) is%
\begin{equation*}
y_{t}=\left( 0,\ldots ,0,e^{\lambda t}\right) ^{T}\text{.}
\end{equation*}%
Again by (\ref{EqXpvar})%
\begin{equation*}
\sup_{t\in \left[ 0,1\right] }\left\vert y_{t}\right\vert =e^{\lambda
}=e^{c_{2}\left\Vert \mathbf{x}\right\Vert _{p\text{-var}}^{p}}
\end{equation*}%
and the upper bound of estimate (\ref{EqLinearVfEstimate}) is attained.
\end{proof}

\bibliographystyle{aalpha}
\bibliography{acompat,rpath}

\end{document}